\documentclass[a4paper,12pt]{amsart}

\usepackage{amsmath,amssymb,amsthm,graphicx,pstricks,url}
\usepackage[all]{xy}
\usepackage{lscape}
\usepackage{hyperref}
\usepackage{color}
\usepackage{tikz}
\usetikzlibrary{arrows}
 \usetikzlibrary{snakes}
 \usetikzlibrary{shapes}
 
\numberwithin{equation}{subsection}

\title[Algebras of acyclic cluster type]{Algebras of acyclic cluster type:\\ tree type and type $\widetilde{A}$}


\author{Claire Amiot}
\address{Institut de Recherche Math\'ematique Avanc\'ee, 7 rue Ren\'e Descartes, 67084 Strasbourg Cedex, France}
\email{amiot@math.unistra.fr}

\author{Steffen Oppermann}
\address{Institutt for matematiske fag,
NTNU,
7491 Trondheim,
Norway}
\email{steffen.oppermann@math.ntnu.no}

\thanks{Work on this project has been done while both authors were working at NTNU, Trondheim, supported by NFR Storforsk grant no.\ 167130. The first author is partially supported by the ANR project ANR-09-BLAN-0039-02.}

\newcommand{\Hom}{{\sf Hom }}
\newcommand{\End}{{\sf End }}
\newcommand{\Ext}{{\sf Ext }}

\newcommand{\Ker}{{\sf Ker }}
\newcommand{\Coker}{{\sf Coker }}

\renewcommand{\mod}{{\sf mod \hspace{.02in}  }}

\newcommand{\gr}{{\sf gr \hspace{.02in} }}
\newcommand{\add}{{\sf add \hspace{.02in} }}

\newcommand{\ten}{\otimes}
\newcommand{\lten}{\overset{\boldmath{L}}{\ten}}

\newcommand{\Jac}{{\sf Jac}}

\newcommand{\Cc}{\mathcal{C}}
\newcommand{\Dd}{\mathcal{D}}

\newcommand{\Hh}{\mathcal{H}}
\newcommand{\Ii}{\mathcal{I}}

\newcommand{\Tt}{\mathcal{T}}
\newcommand{\Mm}{\mathcal{M}}

\newcommand{\Pp}{\mathcal{P}}\newcommand{\Rr}{\mathcal{R}}

\newcommand{\SSS}{\mathbb{S}_2}

\newcommand{\A}{\widetilde{A}}
\newcommand{\bsm}{\begin{smallmatrix}}
\newcommand{\esm}{\end{smallmatrix}}

\newcommand{\ZZ}{\mathbb{Z}}

\newtheorem{thma}{Theorem}[section]
\newtheorem*{thm*}{Th{\'e}or{\`e}me}
\newtheorem{lema}[thma]{Lemma}
\newtheorem*{lem*}{Lemme}

\newtheorem{cora}[thma]{Corollary}
\newtheorem*{prop*}{Proposition}
\newtheorem{prop}[thma]{Proposition}
\theoremstyle{remark}

\newtheorem{rema}[thma]{Remark}

\theoremstyle{definition}

\newtheorem{dfa}[thma]{Definition}
\newtheorem{Question}[thma]{Question}
\numberwithin{figure}{section}

\newcommand{\cov}[2]{{\rm Cov}(#1 , #2)}

\newcommand{\iso}{\cong}
\newcommand{\rL}{{\rm L}}
\newcommand{\rR}{{\rm R}}
\begin{document}
\begin{abstract}
In this paper, we study algebras of global dimension at most 2 whose generalized cluster category is equivalent to the cluster category of an acyclic quiver which is either a tree or of type $\widetilde{A}$. We are particularly interested in their derived equivalence classification. We prove that each algebra which is cluster equivalent to a tree quiver is derived equivalent to the path algebra of this tree. Then we describe explicitly the algebras of cluster type $\A_n$ for each possible orientation of $\A_n$. We give an explicit way to read off in which derived equivalence class such an algebra lies, and describe the Auslander-Reiten quiver of its derived category. Together, these results in particular provide a complete classification of algebras which are cluster equivalent to tame acyclic quivers.
\end{abstract}
\maketitle
\tableofcontents

\section{Introduction}
The classification of finite dimensional algebras over an algebraically closed field $k$ up to derived equivalence is a crucial problem in representation theory. It has a complete answer for algebras of global dimension 1 (see \cite[Corollary~4.8]{Happel87}): Two finite dimensional $k$-algebras $\Lambda=kQ$ and $\Lambda'=kQ'$ are derived equivalent if and only if one can pass from the quiver $Q$ to the quiver $Q'$ by a sequence of reflections (as introduced in \cite{BGP}). Therefore it is possible to decide when two hereditary algebras are derived equivalent by simple combinatorial means. The aim of this paper is to apply results of \cite{AO10} in order to generalize this result to certain algebras of global dimension~2.

The notion of reflection of an acyclic quiver has been generalized by Fomin and Zelevinsky \cite{FZ1} to the notion of mutation in their definition of cluster algebras. Since then,  categorical interpretations of the mutation have been discovered via 2-Calabi-Yau triangulated categories. These created a link between cluster algebras and representation theory. First, in~\cite{BMRRT},  cluster categories $\Cc_Q$ associated to  acyclic quivers $Q$ were defined as the orbit categories $\Dd^b(kQ)/\SSS$, where $\SSS=\mathbb{S}[-2]$ is the second desuspension of the Serre functor $\mathbb{S}$ of the bounded derived category $\Dd^b(kQ)$. This notion has been generalized in \cite{Ami09} to algebras of global dimension two. In this case the generalized cluster category is defined to be the triangulated hull in the sense of \cite{Kel05} of the orbit category $\Dd^b\Lambda/\SSS$. 

\medskip

In this paper, we study more explicitly  the derived equivalence classes of algebras of global dimension 2 which are of acyclic cluster type, that is algebras whose generalized cluster category is equivalent to some cluster category $\Cc_Q$ where $Q$ is an acyclic quiver.  We strongly use the results and the techniques of \cite{AO10}. In particular we use the notion of graded mutation of a graded quiver with potential which is a refinement of  the notion of mutation of a quiver with potential introduced in~\cite{DWZ}: Associated to an algebra $\Lambda$ there is a graded Jacobian algebra $\overline{\Lambda}$~\cite{Kel09} whose degree zero subalgebra is $\Lambda$. Graded mutation explains how to mutate such graded Jacobian algebras. Then from \cite{AO10} we deduce an analogue of the result for algebras of global dimension 1.

\begin{thma}[see Theorem~\ref{thm_gradingQ_derivedeq}]
Let $\Lambda_1$ and $\Lambda_2$ be two finite dimensional algebras of global dimension $2$. Assume that $\Lambda_1$ is of acyclic cluster type.  Then the algebras $\Lambda_1$ and $\Lambda_2$ are derived equivalent if and only if one can pass from $\overline{\Lambda}_1$ to $\overline{\Lambda}_2$ using a sequence of graded mutations.
\end{thma}

The setup is especially nice when the algebras are of tree cluster type. 
\begin{thma}[Corollary~\ref{cor_tree}]
Let $Q$ be an acyclic quiver whose underlying graph is a tree. If $\Lambda$ is an algebra of global dimension 2 of cluster type $Q$, then it is derived equivalent to $kQ$.
\end{thma}

To get a complete understanding of algebras of tame acyclic cluster type, in the rest of the paper, we focus on the algebras of cluster type $\widetilde{A}_{p,q}$. To such an algebra, using graded mutation, we associate an integer that we call weight, which is zero when $\Lambda$ is hereditary. We prove that two algebras of cluster type $\widetilde{A}_{p,q}$ are derived equivalent if and only if they have the same weight (Theorem~\ref{derivedeqiffwegal}). Then a result of \cite{AO10} which shows that two cluster equivalent algebras are graded derived equivalent permits us to compute explicitly the shape of the Auslander-Reiten quiver of the derived category.

\begin{thma}[Corollary~\ref{shapeARquiver}]
 Let $\Lambda$ be an algebra of cluster type $\A_{p,q}$ and of weight  $w\neq 0$. Then the algebra $\Lambda$ is representation finite and not piecewise hereditary. The Auslander-Reiten quiver of $\Dd^b(\Lambda)$ has exactly $3|w|$ connected components:
\begin{itemize}
\item $|w|$ components of type $\ZZ A_\infty^\infty$;
\item $2|w|$ components of type $\ZZ A_\infty$.
\end{itemize}
\end{thma}

Finally we use the explicit description of the cluster-tilted algebras of type $\widetilde{A}_{p,q}$ of \cite{Bas09} to deduce a description of all algebras of cluster type $\widetilde{A}_{p,q}$ in term of quivers with relations.

\medskip

The paper is organized as follows: Section \ref{section background} is devoted to recalling results on generalized cluster categories, on Jacobian algebras, and on 2-Calabi-Yau categories. In Section \ref{section tree case}, we apply the results of \cite{AO10} to algebras of acyclic cluster type. The special case of algebras of cluster type $\widetilde{A}_{p,q}$ is treated in the last three sections. We introduce the notion of weight and prove that it is an invariant of the derived equivalence class of the algebra in Section \ref{section derived atilde}. We compute the shape of the AR-quiver of the derived category in Section~\ref{section ARquiver}, and we finally describe these algebras explicitly in Section \ref{section explicit}. We end the paper by giving (as an example) the complete classification for $\widetilde{A}_{2,2}$.

\subsection*{Acknowledgment}
The authors would like to thank an anonymous referee for his valuable comments and for pointing us out a mistake in a previous version of this paper.

\section{Background}\label{section background}

Throughout this paper, $k$ denotes an algebraically closed field. All categories appearing are $k$-categories, and all functors are $k$-linear. By an algebra we mean an associative unitary basic $k$-algebra. For an algebra $\Lambda$ we denote by $\mod \Lambda$ the category of finitely generated right modules over $\Lambda$. We denote by $D$ the standard duality $\Hom_k(-, k) \colon (\mod k)^{\rm op} \to \mod k$.

\subsection{Cluster-tilting objects and mutation in $2$-Calabi-Yau categories}

\begin{dfa}[Iyama]
Let $\Tt$ be a Krull-Schmidt triangulated category, with finite dimensional $\Hom$-spaces ($\Hom$-finite for short) and $2$-Calabi-Yau, that is there is a functorial isomorphism $\Hom_\Tt(X,Y[2])\iso D\Hom_\Tt(Y,X)$ for all objects $X$ and $Y$ in $\Tt$. An object $T$ is called \emph{cluster-tilting} if 
$$\add(T) =\{X\in \Tt \mid \Hom_\Tt(X,T[1])=0\},$$
where $\add(T)\subset \Tt$ is the additive closure of $T$.
\end{dfa} 

The endomorphism algebra of a cluster-tilting object $T\in \Tt$ is called a \emph{$2$-Calabi-Yau-tilted algebra}.

\begin{thma}[{\cite[Theorem 5.3]{IY}}]\label{IyamaYoshino}
 Let $T$ be a basic cluster-tilting object in a $2$-Calabi-Yau triangulated category $\Tt$. Let $T_i$ be an indecomposable summand of $T\iso T_i\oplus T'$. Then there exists a unique (up to isomorphism) object $T_i^*$ not isomorphic to $T_i$, such that $\mu_{T_i}(T):=T'\oplus T^*_i$ is a basic cluster-tilting object in $\Tt$.  Moreover $T^*_i$ is indecomposable, and there exist triangles in $\Tt$ \[\xymatrix{T_i\ar[r] ^u & U\ar[r] & T^*_i\ar[r] & T_i[1]}\textrm{ and }\xymatrix{T^*_i\ar[r] & U'\ar[r]^{u'} & T_i\ar[r] & T^*_i[1]}\]  where $u$ (resp.\ $u'$) is a minimal left  (resp.\ right) $\add(T')$-approximation. 
\end{thma}

\subsection{Generalized cluster categories}

Let $\Lambda$ be a finite dimensional $k$-algebra of global dimension
at most $2$. We denote by $\Dd^b(\Lambda)$ the bounded derived category
of finitely generated $\Lambda$-modules. It has a Serre functor that
we denote $\mathbb{S}$. We denote by $\SSS$ the composition
$\mathbb{S}[-2]$.

The generalized cluster category $\Cc_\Lambda$ of $\Lambda$ has been defined in \cite{Ami09} as
the triangulated hull of the orbit category $\Dd^b(\Lambda)/\SSS$ (see \cite{Kel05} or \cite[Section~7]{AO10} for more details on triangulated hulls). We
will denote by $\pi_\Lambda$ the triangle functor

$$\xymatrix{\pi_\Lambda\colon \Dd^b(\Lambda)\ar@{->>}[r] &
  \Dd^b(\Lambda)/\SSS\ar@{^(->}[r] & \Cc_\Lambda}$$

We set $\overline{\Lambda}:=\End_\Cc(\pi \Lambda)=\bigoplus_{p\in \mathbb{Z}}\Hom_\Dd(\Lambda,\SSS^{-p}\Lambda).$ By definition this algebra is naturally endowed with a $\mathbb{Z}$-grading.

If $\Lambda=kQ$ is the path algebra of an acyclic quiver, then the cluster category $\Cc_\Lambda=\Cc_Q$ has been introduced in \cite{BMRRT}. We call it  \emph{acyclic cluster category} in this paper.
\begin{dfa}
 A finite dimensional $k$-algebra $\Lambda$ of global dimension $\leq 2$ is said to be \emph{$\tau_2$-finite} if the algebra $\End_\Cc(\pi \Lambda)=\bigoplus_{p\in \mathbb{Z}}\Hom_\Dd(\Lambda,\SSS^{-p}\Lambda)$ is finite dimensional.
\end{dfa}

\begin{thma}[{\cite[Theorem~4.10]{Ami09}}]\label{clustertilting}
Let $\Lambda$ be a finite dimensional algebra of global dimension
$\leq 2$ which is $\tau_2$-finite. Then $\Cc_\Lambda$ is a $\Hom$-finite,
$2$-Calabi-Yau category and the object $\pi(\Lambda)\in \Cc_\Lambda$ is a cluster-tilting object.
\end{thma}

\subsection{Jacobian algebras and 2-Calabi-Yau-tilted algebras}

Quivers with potentials and the associated Jacobian algebras have been studied in \cite{DWZ}.
Let $Q$ be a finite quiver. For each arrow $a$ in $Q$, the \emph{cyclic derivative} $\partial_a$ with
respect to $a$ is the unique linear map
$$\partial_a \colon kQ\rightarrow kQ$$
which sends a path $p$ to the sum $\sum_{p=uav}vu$ taken over all decompositions of the 
path $p$ (where $u$ and $v$ are possibly idempotent elements $e_i$ associated to  a vertex $i$).
A \emph{potential} on $Q$ is any (possibly infinite) linear combination $W$ of cycles 
in $Q$. The associated Jacobian algebra is
$$\Jac(Q,W):=k\hat{Q}/\langle \partial_aW\mid a\in Q_1\rangle,$$
where $k\hat{Q}$ is the completed path algebra (that is the completion of the path algebra $kQ$ at the ideal generated by the arrows of $Q$) and $\langle\partial_aW\mid a\in Q_1\rangle$ is the closure of the ideal generated by $\partial_a W$ for $a\in Q_1$.

Associated with any quiver with potential $(Q,W)$, a cluster category $\Cc(Q,W)$ is constructed in \cite{Ami09}. This construction uses the notion of Ginzburg dg algebras. We refer the reader to \cite{Ami09} for explicit details. When the associated Jacobian algebra is finite dimensional, the category $\Cc(Q,W)$ is $2$-Calabi-Yau and endowed with a canonical cluster-tilting object $T_{(Q,W)}$ whose endomorphism algebra is isomorphic to $\Jac(Q,W)$. The next result gives a link between cluster categories associated with algebra of global dimension $\leq 2$ and cluster categories associated with quiver with potential.

\begin{thma}[{\cite[Theorem 6.12 \textit{a)}]{Kel09}}] \label{keller}
Let $\Lambda=kQ/I$ be a $\tau_2$-finite algebra of global dimension $\leq 2$, such that $I$
is generated by a finite minimal set of relations $\{ r_i \}$. (By this we mean that the set $\{r_i\}$ is the disjoint union of sets representing a basis of the $\Ext_{\Lambda}^2$-space between any two simple $\Lambda$-modules.) The relation
$r_i$ starts at the vertex $s(r_i)$ and ends at the vertex
$t(r_i)$. 
Then there is a triangle equivalence:
$$\Cc_{\Lambda}\iso \Cc(\overline{Q},W),$$ where the quiver $\overline{Q}$ is the quiver $Q$ with additional
arrows $a_i \colon t(r_i)\rightarrow s(r_i)$, and the potential $W$ is
$\sum_i a_ir_i$. This equivalence sends the cluster-tilting object $\pi(\Lambda)$ on the cluster-tilting object $T_{(\overline{Q},W)}$.

As a consequence we have an isomorphism of algebras:
$$\End_\Cc(\pi \Lambda) \iso \Jac(\overline{Q},W).$$ \end{thma}

\begin{dfa}
A potential $W$ on a quiver $Q$ is said to be \emph{rigid} if any cycle $p$ of $Q$ is cyclically equivalent to an element in the Jacobian ideal $\langle \partial_aW\mid a\in Q_1\rangle$.
\end{dfa}
By \cite{DWZ} rigidity is stable under mutation, and it implies that the Gabriel quiver of the Jacobian algebra has no loops nor 2-cycles. 

We end this section with two results linking the mutation of quivers with potential and mutation of cluster-tilting objects in cluster categories.

The first result links the mutation of cluster-tilting objects in the acyclic cluster category and the mutation of rigid quivers with potential defined in \cite{DWZ}.
Acyclic cluster categories are equivalent to stable categories of some Frobenius categories associated to a certain reduced expression in the corresponding Coxeter group \cite[Theorem II.3.4]{BIRSc}. By \cite[Corollary 6.7]{BIRSm} these categories are all \emph{liftable} (see \cite[Section 5]{BIRSm} for definition). Therefore  \cite[Corollary 5.4(b)]{BIRSm} implies the following.
\begin{thma}[Buan-Iyama-Reiten-Smith] \label{birs}
Let $\Delta$ be an acyclic quiver, and $T$ be a basic cluster-tilting object in the cluster category $\Cc_{\Delta}$.  Assume that there exists a quiver with rigid potential $(Q,W)$ with an isomorphism
\[ \End_{\Cc_{\Delta}}(T)\iso \Jac(Q,W). \]
Let $i$ be a vertex of $Q$ and denote by $T_i$ the indecomposable summand of $T\iso T_i\oplus T'$ corresponding to $i$. Then there is an isomorphism
\[ \End_{\Cc_{\Delta}}(\mu_{T_i}(T))\iso \Jac(\mu_i(Q,W)), \]
where $\mu_{T_i}(T)$ is defined in Theorem~\ref{IyamaYoshino} and where $\mu_i(Q,W)$ is the mutation at $i$ of $(Q,W)$ as defined in \cite{DWZ} (see also Subsection~\ref{subsection graded mutation} for definition).
\end{thma}

The second result \cite[Theorem 3.2]{KY11} gives, for two quivers with potential linked by a mutation, an equivalence between the associated cluster categories.

\begin{thma}[Keller-Yang]\label{kelleryang}
Let $(Q,W)$ be a quiver with rigid potential whose Jacobian algebra is finite dimensional, and $i\in Q_0$ a vertex. Then there exists a triangle equivalence $$\Cc(\mu_i(Q,W))\iso \Cc (Q,W)$$ sending the cluster-tilting object $T_{\mu_i(Q,W)}\in \Cc(\mu_i(Q,W))$ onto the cluster-tilting object $\mu_{T_i}(T_{(Q,W)})\in \Cc(Q,W)$, where $T_i$ is the indecomposable summand of $T_{(Q,W)}$ associated with the vertex $i$ of $Q$.

As a consequence we get an isomorphism of algebras 
$$\End_{\Cc(Q,W)}(\mu_{T_i}(T_{(Q,W)})\iso \Jac (\mu_i(Q,W)).$$
\end{thma}

For a sequence $s=(i_1,\ldots, i_r)$ of vertices of $Q$, we denote by $\mu_s$ the composition $\mu_{i_r}\circ\mu_{i_{r-1}}\circ\cdots\circ \mu_{i_1}$ and by $\mu_{s^-}$ the composition $\mu_{i_1}\circ\mu_{i_{2}}\circ\cdots\circ \mu_{i_r}$.

For a bijection  between $\{1,\ldots, n\}$ and the indecomposable summands of a basic cluster-tilting object $T$, we obtain a natural bijection  between $\{1,\ldots, n\}$ and the indecomposable summands  of $\mu_i(T)$. Therefore, once such bijection is fixed,   we will also use the notations  $\mu_s$ and $\mu_{s^-}$ for mutating cluster-tilting objects.

\subsection{Classical results on acyclic cluster categories}

In this paper we are interested in algebras of global dimension $\leq 2$ whose generalized cluster category is equivalent to the cluster category associated to an acyclic quiver. Since we will strongly make use of them, we now recall some results for acyclic cluster categories.

The following theorem  follows from a result by Happel and Unger \cite{HU03}.
\begin{thma}[{\cite[Prop. 3.5]{BMRRT}} -- see also \cite{Hub08}] \label{connectivity}
Let $Q$ be an acyclic quiver, and let $T$ be a cluster-tilting object of $\Cc_Q$. Then there exists a sequence of mutations linking the cluster-tilting object $T$ to the canonical cluster-tilting object $\pi_Q(kQ)$. In other words, the cluster-tilting graph of the acyclic cluster category is connected.
\end{thma}

A consequence of this theorem together with Theorem~\ref{birs} is that the endomorphism algebra of a cluster-tilting object in an acyclic cluster category is always a Jacobian algebra of a rigid QP.

Another special feature of acyclic cluster categories is that they can be recognized by the quivers of their cluster-tilting objects.

\begin{thma}[\cite{KR08}] \label{recognition}
Let $\Cc$ be an algebraic triangulated category, which is $\Hom$-finite and $2$-Calabi-Yau. If there exists a cluster-tilting object $T\in\Cc$ such that $\End_\Cc(T)\iso kQ$, where $Q$ is an acyclic quiver, then there exists a triangle equivalence $\Cc\iso \Cc_Q$ sending $T$ onto $\pi_Q(kQ)$. 

As a consequence there is a triangle equivalence $\Cc(Q,0)\to \Cc_Q$ sending $T_{(Q,0)}$ onto $\pi_Q(kQ)$.
\end{thma}

Note that an analogue of these results is not known for generalized cluster categories. 

An algebra of the form $\End_{\Cc_Q}(T')$, where $T'$ is a cluster-tilting object in $\Cc_Q$, is called a \emph{cluster-tilted algebra of type Q}.

\section{Derived equivalent algebras of tree cluster type}\label{section tree case}
In this section we investigate algebras of global dimension at most 2 whose generalized cluster category is a cluster category associated to a tree. We will strongly use results from~\cite{AO10}.

We start with a definition.
\begin{dfa}(\cite[Def. 5.1]{AO10})
Two algebras $\Lambda_1$ and $\Lambda_2$  of global dimension $\leq 2$ which are $\tau_2$-finite are said to be \emph{cluster equivalent} if there exists a triangle equivalence $\Cc_{\Lambda_1}\rightarrow \Cc_{\Lambda_2}$ between their generalized cluster categories.

If $\Lambda$ is cluster equivalent to $kQ$ where $Q$ is an acyclic quiver, we will say that $\Lambda$ is of \emph{cluster type $Q$}.
\end{dfa}

Two derived equivalent algebras of global dimension $\leq 2$ are cluster equivalent (\cite[Cor.~7.16]{AO10}). Hence, if  the underlying graph of $Q$ is a tree, then the class of algebras of cluster type $Q$ does not depend on the orientation of $Q$.

 From the results of Section~\ref{section background} we deduce the following characterization of algebras of acyclic cluster type.
 
 \begin{cora}\label{recognitioncor}
 Let $\Lambda=kQ/I$ be a $\tau_2$-finite algebra of global dimension $\leq 2$. Let $(\overline{Q},W)$ the associated quiver with potential defined in Theorem~\ref{keller}. Then $\Lambda$ is of acyclic cluster type $\Delta$ if and only if there exists a sequence of mutation $s=i_1,\ldots, i_l$ such that $(\Delta,0)=\mu_s(\overline{Q},W)$. 
 
In this case, $(\overline{Q},W)$ is a rigid quiver with potential, and there exists a triangle equivalence $f:\Cc_{\Lambda}\to\Cc_\Delta$ sending $\pi_\Lambda(\Lambda)$ to $\mu_{s^-}(\pi_\Delta(k\Delta))$.
\end{cora}
  
  \begin{proof}
  By Theorems \ref{keller}, \ref{kelleryang}, and \ref{recognition}, if $(\Delta,0)=\mu_s(\overline{Q},W)$, then  we have equivalences
  $$\Cc_{\Lambda}\iso \Cc(\overline{Q},W)\iso \Cc(\Delta,0)\iso \Cc_\Delta$$ sending $\pi_\Lambda(\Lambda)$ to $\mu_{s^-}(\pi_\Delta(k\Delta))$. 
    
Conversely assume that there exists an equivalence $f:\Cc_{\Lambda}\iso \Cc_\Delta$. Then by Theorem~\ref{connectivity} there exists a sequence of mutations $s$ such that $\pi_\Delta(k\Delta)\iso \mu_s f(\pi_\Lambda(\Lambda))$. So by Theorem~\ref{birs} we have $\Jac(\mu_{s}(\overline{Q},W))\iso \Jac(\Delta,0)=k\Delta$. Since the quiver with potential $\mu_s(\overline{Q},W)$ is reduced, we necessarily have $\mu_s(\overline{Q},W)=(\Delta,0)$. 
  
  \end{proof}

\subsection{Graded equivalence and derived equivalence}
Cluster equivalence is strongly related with graded equivalence. In this subsection, we will recall some results shown in \cite{AO10}.

Let $A=\bigoplus_{p\in \ZZ}A^p$ be $\ZZ$-graded algebra. We denote by $d$ the degree map sending any homogeneous element of $A$ to its degree, and by $\gr A$ the category of finite generated graded modules over $A$.  For a graded module $M=\bigoplus_{p\in\ZZ}M^p$, we denote by $M\langle q\rangle$ the graded module $\bigoplus_{p\in\ZZ}M^{p+q}$ (that is, the degree $p$ part of $M\langle q \rangle$ is $M^{p+q}$). The locally bounded subcategory 
\[\cov{A}{d}:=\add\{A\langle p\rangle  \mid p\in\ZZ\} \subseteq \gr A\]
is called the \emph{$\ZZ$-covering} of $A$.

\begin{dfa}
Let $A_1$ and $A_2$ be two $\ZZ$-graded algebras. Assume that $A_1$ and $A_2$ are isomorphic as algebras. We will say that $A_1$ and $A_2$ are \emph{graded equivalent} (and write $A_1\underset{\gr}{\sim} A_2$) if there exist $r_i\in \ZZ$ and an isomorphism of $\ZZ$-graded algebras \[ \xymatrix{A_2\ar[r]^-\sim_-\ZZ & \bigoplus_{p\in \ZZ}\Hom_{\cov{A_1}{\ZZ}}(\bigoplus_{i=1}^nP_i\langle r_i\rangle, \bigoplus_{i=1}^nP_i\langle r_i+p\rangle)}\]
where $A_1=\bigoplus_{i=1}^nP_i$ is a decomposition of $A_1$ into indecomposable graded projective modules. This is equivalent to the fact that the coverings $\cov{A_1}{d_1}$ and $\cov{A_2}{d_2}$ are equivalent as categories with $\mathbb{Z}$-action.
\end{dfa}

The link between cluster equivalent algebras and graded equivalent algebras is given by the following result.

\begin{thma}[{\cite[Thm 5.8]{AO10}}]\label{AO_derivedeq_gradedeq_thm}
Let $\Lambda_1$ and $\Lambda_2$ be two $\tau_2$-finite algebras of global dimension $\leq 2$. For $i=1,2$, denote by $\Dd_i$ the bounded derived category of $\Lambda_i$, by $\Cc_i$ its cluster category, and by $\pi_i$ the natural functor $\Dd_i \to \Cc_i$. Assume there is $T \in \Dd_1$ such that $\pi_1(T)$ is basic cluster-tilting in $\Cc_1$, and
\begin{enumerate}
\item there is an isomorphism $\xymatrix{\End_{\Cc_1}(\pi_1 T)\ar[r]^-\sim & \End_{\Cc_2}(\pi_2\Lambda_2)}$
\item this isomorphism can be chosen in such a way that the two $\ZZ$-gradings defined on $\overline{\Lambda}_2$, given respectively by
$$\bigoplus_{q\in \ZZ}\Hom_{\Dd_2}(\Lambda_2,\SSS^{-q}\Lambda_2)\textrm{ and } \bigoplus_{p\in \ZZ}\Hom_{\Dd_1}(T, \SSS^{-p}T),$$
are equivalent.
\end{enumerate}
Then the algebras $\Lambda_1$ and $\Lambda_2$ are derived equivalent, and hence cluster equivalent.
\end{thma}

Note that the functor $\SSS^{-1}$ acts on the subcategory $\pi_1^{-1}\pi_1(T)=\add\{\SSS^{-q}T, q\in \ZZ\}\subset \Dd^b(\Lambda_1)$. As a category with $\ZZ$-action, it is equivalent to the $\ZZ$-covering of the graded algebra $\End_{\Cc_1}(\pi_1 T)$. Therefore, by the above remark, the graded algebras $\End_{\Cc_1}(\pi_1 T)$ and $\End_{\Cc_2}(\pi_2 \Lambda_2)$ are graded equivalent if and only if there is an equivalence between the additive categories $\pi_1^{-1}\pi_1(T)$ and $\pi_2^{-1}\pi_2(\Lambda_2)$ as categories with $\SSS$-action.

\subsection{Mutation of a graded quiver with potential}\label{subsection graded mutation}

In order to apply Theorem~\ref{AO_derivedeq_gradedeq_thm} we make use of a tool: the left (or right) mutation of a graded quiver with potential which extends the Derksen-Weyman-Zelevinsky mutation of a quiver with potential \cite{DWZ}. All definitions and results described in this subsection are proved in \cite[Section 6]{AO10}. 

\begin{dfa}[{\cite{AO10}}]
Let $(Q,W,d)$ be a $\ZZ$-graded quiver with potential homogeneous
  of degree 1 (graded QP for short). Let $i\in Q_0$ be a vertex, such that there are neither loops
  nor 2-cycles incident to $i$. We define $\mu_i^{\rL}(Q,W,d)$, the \emph{left graded mutation of $(Q,W,d)$ at vertex $i$}, as the reduction of the graded QP $(Q',W',d')$. The quiver $Q'$ is defined as in
  \cite{DWZ} as follows:
\begin{itemize}
\item for any subquiver $\xymatrix{u\ar[r]^a &i\ar[r]^b & v}$, with
  $i$, $u$, and $v$ pairwise distinct vertices, we add an arrow
  $[ba]\colon u\rightarrow v$;
\item we replace all arrows $a$ incident with $i$ by an arrow $a^*$ in
  the opposite direction.
\end{itemize}

The potential $W'$ is also defined as in \cite{DWZ} by the sum $[W]+W^*$ where $[W]$ is formed
from the
potential $W$ replacing all compositions $ba$ through the vertex $i$
by $[ba]$, where $W^*$ is the sum $\sum a^* b^* [ba]$.

The new degree $d'$ is defined as follows:

\begin{itemize}
\item $d'(a)=d(a)$ if $a$ is an arrow of $Q$ and $Q'$;
\item $d'([ba])=d(b)+d(a)$ if $ba$ is a composition passing through
  $i$;
\item $d'(a^*)=-d(a)+1$ if the target of $a$ is $i$;
\item $d'(b^*)=-d(b)$ if the source of $b$ is $i$.
\end{itemize} 
\end{dfa}
One can check that this operation is compatible with the reduction of a QP (see \cite{AO10}). It is possible to define the right graded mutation $\mu_i^{\rR}$ by interchanging target and source in the last two items in the definition. We have the following.
\begin{lema}[\cite{AO10}] \label{leftrightmutation}
Let $(Q,W,d)$ be a graded quiver with potential. Then we have an isomorphism of $\ZZ$-graded algebras \[\Jac(Q,W,d)\underset{\ZZ}{\iso}\Jac(\mu_i^{\rR}\mu_i^{\rL}(Q,W,d)).\]
\end{lema}
Moreover, this mutation preserves graded equivalence. More precisely, we have the following.
\begin{prop}[\cite{AO10}]\label{prop_gradedeq_mutation}
Let $(Q,W,d_1)$ and $(Q,W,d_2)$ be two $\ZZ$-graded QP such that the graded Jacobian algebras $\Jac(Q,W,d_1)$ and $\Jac(Q,W,d_2)$ are graded equivalent. Then the graded Jacobian algebras $\Jac(\mu_i^{\rL}(Q,W,d_1))$ and $\Jac(\mu_i^{\rL}(Q,W,d_2))$ are graded equivalent.
\end{prop}

\begin{dfa}[\cite{AO10}]
Let $\Lambda$ be a $\tau_2$-finite algebra of global dimension $\leq2$. Let $T=T_1 \oplus\cdots  \oplus T_n$ be an object in $\Dd^b(\Lambda)$ such that $\pi(T)$ is a (basic) cluster-tilting object in $\Cc_\Lambda$. Let $T_i$ be an indecomposable summand of $T\iso T'\oplus T_i$. Define $T_i^{\rL}$ as the cone in $\Dd^b(\Lambda)$ of the minimal left $\add\{\SSS^p T',p\in\ZZ\}$-approximation $u\colon T_i\rightarrow B$ of $T_i$. We denote by $\mu_i^{\rL}(T)$ the object $T_i^{\rL}\oplus T'$ and call it the \emph{left mutation of $T$ at $T_i$}.
\end{dfa}
It is immediate to check that we have $\pi(\mu_i^{\rL}(T))=\mu_i(\pi(T))$, thus $\pi(\mu_i^{\rL}(T))$ is a cluster-tilting object in $\Cc_\Lambda$.

This (left) mutation in the derived category is reflected by the graded (left) mutation of graded QP in the following sense.
\begin{thma} \label{Zgradedbirs}
Let $\Lambda$ be an algebra of acyclic cluster type and $T\in\Dd^b(\Lambda)$ as above. Assume that there
exist a $\ZZ$-graded QP $(Q,W,d)$ with rigid potential homogeneous of degree 1 such that we have an
isomorphism of graded algebras 
$$\xymatrix{\bigoplus_{p\in\ZZ}\Hom_{\Dd}(T,\SSS^{-p} T)\ar[r]^-\sim_-{\ZZ} & \Jac(Q,W,d).}$$
Let $i\in Q_0$ and $T_i$ be the associated indecomposable summand of $T\iso T_i\oplus T'$ . 
Then there is an isomorphism of $\ZZ$-graded algebras
$$\xymatrix{\bigoplus_{p\in
  \mathbb{Z}}\Hom_\Dd(T'\oplus T^{\rL}_i,\SSS^{-p}(T'\oplus T^{\rL}_i))\ar[r]^-\sim_-{\ZZ}& \Jac(\mu_i^{\rL}(Q,W,d)).}$$  
\end{thma}

\begin{proof} By Theorem \ref{birs}, we already have an isomorphism between the algebras 
$$\bigoplus_{p\in
  \mathbb{Z}}\Hom_\Dd(T'\oplus T^{\rL}_i,\SSS^{-p}(T'\oplus T^{\rL}_i))\iso \End_{\Cc_\Lambda}(\mu_{\pi(T_i)}(\pi T))$$ and $\Jac(\mu_i^{\rL}(Q,W))$ forgetting the grading. The only thing to check is that it is compatible with the grading. The proof of this fact is the same as the one of Theorem 6.10 in \cite{AO10}.\end{proof}

\begin{dfa}\label{defWgrading}
Let $(Q,W)$ be a reduced quiver with potential. A grading $d$ on $Q$ will be said to be a \emph{$W$-grading} if
\begin{itemize}
\item for all arrows $a$ in $Q_1$, $d(a)\in \{0,1\}$;
\item the potential $W$ is homogeneous of degree $1$;
\item the set of relations $\{\partial_aW \mid d(a)=1\}$ is linearly independent (in particular for all $a\in Q_1$ such that $d(a)=1$, we have $\partial_aW\neq 0$).
\end{itemize}
\end{dfa}

Theorem~\ref{Zgradedbirs} is particularly useful with the following result.

\begin{prop}[\cite{AO10}] \label{propWgrading}
Let $\Lambda=kQ_{\Lambda}/I$ be an algebra of global dimension $\leq 2$ which is $\tau_2$-finite and of acyclic cluster type. Then there exists a rigid quiver with potential $(\overline{Q},W)$ and a $W$-grading $d$ such that we have an isomorphism of graded algebras \[\overline{\Lambda}\underset{\ZZ}{\iso} \Jac(\overline{Q},W,d).\]
\end{prop}
The quiver with potential is given by Theorem~\ref{keller}; the arrows of $Q_{\Lambda}$ have
degree zero and the arrows $a_i$ corresponding to the relations have
degree 1.  The rigidity comes from Corollary \ref{recognitioncor}. 

\subsection{Application to acyclic cluster type}

Let $\Lambda$ be an algebra of cluster type $Q$, where $Q$ is an acyclic quiver. By Corollary \ref{recognitioncor} there exists a sequence of mutations $s$ such that $\mu_s(Q_{\overline{\Lambda}}, W)=(Q,0)$, where $(Q_{\overline{\Lambda}},W)$ is the quiver with potential associated with $\Lambda$ in Theorem \ref{keller}.

\begin{dfa} The map $d_s\colon Q\rightarrow \ZZ$ is called \emph{a grading induced by $\Lambda$} if it satisfies $$\mu^{\rL}_s(Q_{\overline{\Lambda}},W,d)=(Q,0,d_s),$$
where $s$ is a sequence of mutations as in Corollary \ref{recognitioncor}. 
\end{dfa} 

\begin{prop}\label{uniqueness_induced_grading}
Let $d_s$ and $d_t$ be two gradings induced by $\Lambda$ on $Q$. If $\mu_s(\pi_{\Lambda} (\Lambda)) \iso \mu_t(\pi_{\Lambda} (\Lambda))$ then $d_s$ and $d_t$ are equivalent up to automorphism of $Q$.
\end{prop}

\begin{proof}
Let
$\mu^{\rL}_s \Lambda \iso T_1 \oplus \cdots \oplus T_n$
be a decomposition of $\mu^{\rL}_s \Lambda$ into indecomposable summands. Then we have
\[ \mu^{\rL}_t \Lambda \iso \SSS^{r_1} T_1 \oplus \cdots \oplus \SSS^{r_n} T_n \]
for certain $r_i \in \mathbb{Z}$, since $\pi_{\Lambda} \mu^{\rL}_s \Lambda = \mu_s \pi_{\Lambda} \Lambda \iso  \mu_t \pi_{\Lambda} \Lambda = \pi_{\Lambda} \mu^{\rL}_t \Lambda$. It follows that the algebras
\[ \bigoplus_{p \in \ZZ} \Hom_{\Dd^b(\Lambda)}(\mu^{\rL}_s \Lambda, \SSS^{-p} \mu^{\rL}_s \Lambda) \qquad \text{and} \qquad \bigoplus_{p \in \ZZ} \Hom_{\Dd^b(\Lambda)}(\mu^{\rL}_t \Lambda, \SSS^{-p} \mu^{\rL}_t \Lambda) \]
are graded equivalent. Since, by Theorem~\ref{Zgradedbirs}, these graded algebras are isomorphic to $\Jac(Q, 0, d_s)$ and $\Jac(Q, 0, d_t)$, respectively, it follows that the gradings $d_s$ and $d_t$ are equivalent up to automorphisms of $Q$.
\end{proof}

\begin{rema}\label{counter example}
The induced grading depends in general on the choice of a cluster-tilting object in $\Cc_{\Lambda}$ with hereditary endomorphism ring, as shown in the following example. Let $\Lambda$ be the algebra presented by the quiver
\[ Q =
\begin{tikzpicture}[baseline=13pt]
 \node (A) at (0,0) {$1$}; 
 \node (B) at (1,1) {$2$}; 
 \node (C) at (2,0) {$3$}; 
 \node (D) at (3,1) {$4$}; 
 \node (E) at (4,0) {$5$};
 \draw [->] (A) -- node [above left=-2pt] {$\scriptstyle \alpha$} (B);
 \draw [->] (B) -- node [above right=-2pt] {$\scriptstyle \beta$} (C);
 \draw [->] (C) -- (A);
 \draw [->] (D) -- (C);
 \draw [->] (E) -- (C);
 \draw [->] (E) -- (D);
\end{tikzpicture} 
\text{ with relation } \beta \alpha = 0. \]
Adding an arrow of degree $1$ for the relation, one obtains a graded quiver with potential $(\overline{Q}, W, d)$. Then it is easy to check that $\mu^{\rL}_2 (\overline{Q}, W, d)$ is given by the graded quiver
\[ \begin{tikzpicture}
 \node (A) at (0,0) {$1$}; 
 \node (B) at (1,1) {$2$}; 
 \node (C) at (2,0) {$3$}; 
 \node (D) at (3,1) {$4$}; 
 \node (E) at (4,0) {$5$};
 \draw [->] (B) -- node [fill=white, inner sep=0pt] {$\scriptstyle 1$} (A);
 \draw [->] (C) -- node [fill=white, inner sep=0pt] {$\scriptstyle 0$} (B);
 \draw [->] (C) -- node [fill=white, inner sep=0pt] {$\scriptstyle 0$}  (A);
 \draw [->] (D) -- node [fill=white, inner sep=0pt] {$\scriptstyle 0$}  (C);
 \draw [->] (E) -- node [fill=white, inner sep=0pt] {$\scriptstyle 0$}  (C);
 \draw [->] (E) -- node [fill=white, inner sep=0pt] {$\scriptstyle 0$}  (D);
\end{tikzpicture} \]
On the other hand  $\mu^{\rL}_2 \mu^{\rL}_1 \mu^{\rL}_4 \mu^{\rL}_5 \mu^{\rL}_2 (\overline{Q}, W, d)$ is given by the graded quiver
\[ \begin{tikzpicture}[baseline]
 \node (A) at (0,0) {$4$}; 
 \node (B) at (1,1) {$5$}; 
 \node (C) at (2,0) {$3$}; 
 \node (D) at (3,1) {$1$}; 
 \node (E) at (4,0) {$2$};
 \draw [->] (B) -- node [fill=white, inner sep=0pt] {$\scriptstyle 0$} (A);
 \draw [->] (C) -- node [fill=white, inner sep=0pt] {$\scriptstyle 0$} (B);
 \draw [->] (C) -- node [fill=white, inner sep=0pt] {$\scriptstyle 0$}  (A);
 \draw [->] (D) -- node [fill=white, inner sep=0pt] {$\scriptstyle 1$}  (C);
 \draw [->] (E) -- node [fill=white, inner sep=0pt] {$\scriptstyle 1$}  (C);
 \draw [->] (E) -- node [fill=white, inner sep=0pt] {$\scriptstyle 1$}  (D);
\end{tikzpicture} \]
These gradings are not equivalent.
\end{rema}

Using the previous results we can prove the following.
\begin{thma}\label{thm_gradingQ_derivedeq}
Let $\Lambda_1$ and $\Lambda_2$ be two algebras of cluster type $Q$, where $Q$ is an acyclic quiver. Then the following are equivalent.
\begin{enumerate}
\item  There exist equivalent gradings $d_{s_i}$ induced by $\Lambda_i$ on $Q$ for $i \in \{1,2\}$.
\item There exists a derived equivalence $\Dd^b(\Lambda_1)\rightarrow \Dd^b(\Lambda_2)$.
\end{enumerate}
\end{thma}

\begin{proof}
For $i=1,2$ denote by $(Q_i,W_i,\partial_i)$ the graded QP associated to $\overline{\Lambda}_i$. 

$(1)\Rightarrow (2)$ By assumption, we have a graded equivalence \[\Jac(\mu_{s_1}^{\rL}(Q_1,W_1,\partial_1))\underset{\gr}{\sim}\Jac(\mu_{s_2}^{\rL}(Q_2,W_2,\partial_2)).\] Then by Proposition~\ref{prop_gradedeq_mutation} and Lemma~\ref{leftrightmutation}, we immediately get that \[\Jac(\mu_{s_2^-}^{\rR}\mu_{s_1}^{\rL}(Q_1,W_1,\partial_1))\underset{\gr}{\sim} \Jac(Q_2,W_2,\partial_2).\]
Now denote by $T\in \Dd^b(\Lambda_1)$ the object $\mu_{s_2^-}^{\rR}\mu_{s_1}^{\rL}(\Lambda_1)$. By Theorem~\ref{Zgradedbirs} we have an isomorphism of $\ZZ$-graded algebras
\[ \bigoplus_{p\in\ZZ}\Hom_{\Dd^b(\Lambda_1)}(T,\SSS^{-p}T)\underset{\ZZ}{\iso} \Jac(\mu_{s_2^-}^{\rR}\mu_{s_1}^{\rL}(Q_1,W_1,\partial_1))\]
Therefore, we get the result by Theorem~\ref{AO_derivedeq_gradedeq_thm}.

$(2)\Rightarrow (1)$ Assume that $\Lambda_1$ and $\Lambda_2$ are derived equivalent. Then there exists a tilting complex $T$ in $\Dd^b(\Lambda_2)$ with $\End_{\Dd^b(\Lambda_2)}(T)\iso \Lambda_1$. The derived equivalence induced by $T$ gives rise to a commutative diagram \[\xymatrix{\Dd^b(\Lambda_1)\ar[d]_{\pi_1}\ar[rr]^{-\lten_{\Lambda_1}T}_\sim && \Dd^b(\Lambda_2)\ar[d]^{\pi_2}\\ \Cc_{\Lambda_1}\ar[rr]_f^{\sim}&& \Cc_{\Lambda_2}}\]
Since the cluster category $\Cc_{\Lambda_2}$ is acyclic, there exists a sequence of mutations $s$ such that $\mu_s(\pi_2\Lambda_2)=\pi_2(T)=f(\pi_1\Lambda_1).$ Denote by $T'\in \Dd^b(\Lambda_2)$ the object $T':=\mu_s^{\rL}(\Lambda_2).$ Then we have $\pi_2(T)=\pi_2(T')$, thus we have an equivalence of categories commuting with $\SSS$
\[\add\{\SSS^pT \mid p\in\ZZ\}\iso \add\{\SSS^p T'\mid p\in\ZZ\}.\] This exactly means that there is a graded equivalence
\[\bigoplus_{p\in\ZZ}\Hom_{\Dd^b(\Lambda_2)}(T,\SSS^{-p}T)\underset{\gr}{\sim}\bigoplus_{p\in\ZZ}\Hom_{\Dd^b(\Lambda_2)}(T',\SSS^{-p}T').\] 
Since $-\lten_{\Lambda_1}T$ is an equivalence, the left term is isomorphic as $\ZZ$-graded algebra to $\Jac(Q_1,W_1,\partial_1)$ and by Theorem~\ref{Zgradedbirs} the right term is isomorphic to the $\ZZ$-graded algebra $\Jac(\mu_s^{\rL}(Q_2,W_2, \partial_2))$. Therefore we have \begin{equation}\label{eqiso} \Jac(\mu_s^{\rL}(Q_2,W_2,\partial_2))\underset{\gr}{\sim} \Jac(Q_1,W_1,\partial_1).\end{equation}
Now let $s'$ be a sequence such that the cluster-tilting object $T'':=\mu_{s'}(f\pi_1\Lambda_1)$ has endomorphism algebra isomorphic to $kQ$. Then we have \[\mu_{s'}(Q_1,W_1)=(Q,0)\textrm{ and } \mu_{s'}\mu_{s}(Q_2,W_2)=(Q,0).\]
Now let $d_1$ and $d_2$ be the gradings on $Q$ such that we have 
\[\mu_{s'}^{\rL}(Q_1,W_1,\partial_1)=(Q,0,d_1)\textrm{ and } \mu^{\rL}_{s'}\mu^{\rL}_{s}(Q_2,W_2,\partial_2)=(Q,0,d_2).\] By \eqref{eqiso} and Proposition~\ref{prop_gradedeq_mutation}, we get 
\[ \Jac(Q,0,d_1)\underset{\gr}{\sim} \Jac(Q, 0,d_2), \]
that is the gradings $d_1$ and $d_2$ are equivalent.
\end{proof}

\begin{cora}\label{cor_tree}
If $Q$ is a tree, and if $\Lambda$ is of cluster type $Q$, then $\Lambda$ is derived equivalent to $kQ$.
\end{cora}

\begin{proof}
By Theorem~\ref{thm_gradingQ_derivedeq}, it is enough to observe that all gradings on $kQ$ are equivalent. We consider the map
\[ \begin{array}{rcl}\ZZ^{Q_0} & \rightarrow & \ZZ^{Q_1} \\ (r_i)_{i\in Q_0} & \mapsto & (r_{t(a)}-r_{s(a)})_{a\in Q_1}\end{array}. \]
Then the grading on $Q$ which are equivalent to the trivial grading are the image of this map. Moreover, using the fact that $Q$ is a tree, one easily sees that the map is surjective.
\end{proof}

\section{Derived equivalence classes of algebras of cluster type $\A_{p,q}$}\label{section derived atilde}

In this section we study the algebras of cluster type $\widetilde{A}_{p,q}$, i.e.\ all algebras which are cluster equivalent to the path algebra $H$ of the acyclic quiver $Q_H$: 
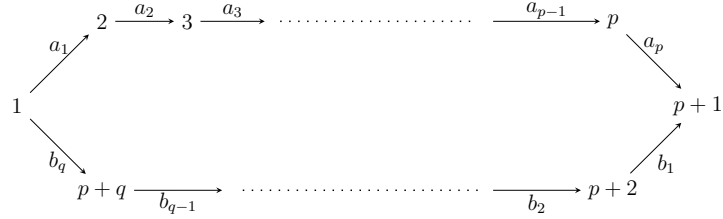
\[\scalebox{.7}{
\begin{tikzpicture}[>=stealth,scale=1.6]
\node (P1) at (0,0) {$1$};
\node (P2) at (1,1) {$2$};
\node (P3) at (2,1) {$3$};
\node (P4) at (3,1){};
\node (P5) at (5.5,1){};
\node (P6) at (7,1) {$p$};
\node (P7) at (8,0) {$p+1$};
\node (P8) at (7,-1) {$p+2$};
\node (P9) at (5.5,-1) {};
\node (P10) at (2.5,-1) {};
\node (P11) at (1,-1) {$p+q$};
\draw [->] (P1)  -- node [swap,yshift=3mm] {$a_1$}(P2);
\draw [->] (P2) -- node [swap,yshift=2mm] {$a_2$}(P3);
\draw [->] (P3) -- node [swap,yshift=2mm] {$a_3$}(P4);
\draw [->] (P5) --node [swap,yshift=2mm] {$a_{p-1}$} (P6);
\draw [->] (P6) -- node [swap,yshift=3mm] {$a_p$}(P7);
\draw [->] (P8) -- node [swap,yshift=-3mm,xshift=2mm] {$b_1$}(P7);
\draw [->] (P9) -- node [swap,yshift=-2.5mm] {$b_2$}(P8);
\draw [->] (P11) -- node [swap,yshift=-2.5mm] {$b_{q-1}$}(P10);
\draw [->] (P1) -- node [swap,yshift=-3mm] {$b_q$}(P11);
\draw [loosely dotted, thick] (P4) -- (P5);
\draw [loosely dotted, thick] (P9) --(P10);
\end{tikzpicture}}
\]

 Since $Q_H$ is not a tree, these algebras need not be derived equivalent. In this section we introduce an invariant of an algebra of cluster type $Q_H$ which determines its class of derived equivalence.

\subsection{The weight of an algebra of cluster type $\A_{p,q}$} 

\begin{dfa}
Let $d$ be a $\ZZ$-grading on $Q_H$. We define the weight of the grading $d$ by
\[w(H,d):=\sum_{i=1}^pd(a_i)-\sum_{i=1}^qd(b_i)\]
\end{dfa}

\begin{lema}\label{gradedeqandw}
Let $d_1$ and $d_2$ be two $\mathbb{Z}$-gradings on $H$. Then the following are equivalent:
\begin{enumerate}
\item $(H,d_1)$ and $(H,d_2)$ are graded equivalent;
\item $w(H,d_1)=w(H,d_2)$.
\end{enumerate}
\end{lema}

\begin{proof}
 We consider the map
$$\begin{array}{rrcl}\phi\colon &\ZZ^{(Q_H)_0} & \rightarrow & \ZZ^{(Q_H)_1} \\& (r_i)_{i\in (Q_H)_0} & \mapsto & (r_{t(a)}-r_{s(a)})_{a\in (Q_H)_1}\end{array}.$$
The gradings $d_1$ and $d_2$ are equivalent if and only if $d_1-d_2$ is in the image of $\phi$.

It is straight-forward to check that the sequence of $\ZZ$-modules
\[ \ZZ^{(Q_H)_0} \overset{\phi}{\rightarrow} \ZZ^{(Q_H)_1} \overset{w}{\rightarrow} \ZZ \rightarrow 0 \]
is exact (using the fact that $Q$ is of type $\widetilde{A}$). Now the claim of the lemma follows.
\end{proof}

This result allows us to define the following.
\begin{dfa}
Let $\Lambda$ be an algebra of global dimension at most 2 and of cluster type $\A_{p,q}$. Let $s$ be a sequence of mutations such that the endomorphism ring of $\mu_s(\pi_{\Lambda}(\Lambda))$ is $kQ_H$ (such a sequence exists by Corollary \ref{recognitioncor}), and let $d_s$ be the corresponding grading on $Q_H$ induced by $\Lambda$. Then define the \emph{weight} of $\Lambda$ by
\[w(\Lambda,s):=w(d_s)=\sum_{i=1}^pd_s(a_i)-\sum_{i=1}^qd_s(b_i).\]
\end{dfa}
If $p>q$, the weight $w(\Lambda,s)$ is well defined since there is no automorphism of $Q_H$.

\subsection{The case $p>q$}
Before proving the main theorem, we prove the following technical result.

\begin{lema} \label{lem.w_well-def}
The weight $w(\Lambda,s)$ does not depend on the choice of the sequence of mutations $s$.
\end{lema}

\begin{proof}
Let $s$ and $t$ be two sequences of mutations such that the endomorphism rings of $\mu_s( \pi_{\Lambda}(\Lambda))$ and of $\mu_t(\pi_{\Lambda} (\Lambda))$ are both isomorphic to $kQ_H$. Then there is an automorphism of the translation quiver $\ZZ Q_H$, which forms a component of the Auslander-Reiten quiver of the common cluster category, mapping $\mu_s( \pi_{\Lambda}(\Lambda))$ to $\mu_t(\pi_{\Lambda} (\Lambda))$.
Since $p\neq q$, one can check that the automorphism group of the translation quiver $\ZZ Q_H$ is generated by 2 elements given by the slices
\begin{align*}
T_1 & =\tau^{-1}(e_1H\oplus\cdots\oplus e_pH)\oplus e_{p+q}H\oplus\cdots \oplus e_{p+1}H\text{, and} \\
T_2 & =e_2H\oplus\cdots\oplus e_{p+1}H\oplus \tau^{-1}(e_1H\oplus e_{p+q}H\oplus e_{p+q-1}H\oplus\cdots\oplus e_{p+2}H)\text{, respectively,}
\end{align*}
where $e_i$ is the primitive idempotent of $H=kQ$ associated to the vertex $i$ for any $i=1,\ldots, p+q$.
Then we have
$$T_1=\mu^{\rL}_p\mu^{\rL}_{p-1}\ldots\mu_1^{\rL}(H) \text{ and } T_2=\mu_{p+2}^{\rL}\mu_{p+3}^{\rL}\ldots\mu_{p+q}^{\rL}\mu_{1}^{\rL}(H).$$
It is enough to check that if $d$ is a grading on $Q_H$, then the weights of the graded quivers $\mu^{\rL}_p\mu^{\rL}_{p-1}\ldots\mu_1^{\rL}(Q_H,d)$ and $\mu_{p+2}^{\rL}\mu_{p+3}^{\rL}\ldots\mu_{p+q}^{\rL}\mu_{1}^{\rL}(Q_H,d)$ are equal to $w(d)$. We do this for  $\mu^{\rL}_p\mu^{\rL}_{p-1}\ldots\mu_1^{\rL}(Q_H,d)$, the other one is similar.

First note that in the sequence of mutations $\mu^{\rL}_p\mu^{\rL}_{p-1}\ldots\mu_1^{\rL}$, we mutate at a source at each step. Therefore the left graded mutation consists in reversing the arrows and assigning the opposite of the degree. After the sequence of mutations, arrows $a_1,\ldots,a_{p-1}$ have been reversed twice, arrows $a_p$ and $b_q$ have been reversed once, and $b_1,\ldots, b_q$ have not been reversed. Hence the graded quiver $\mu^{\rL}_p\mu^{\rL}_{p-1}\ldots\mu_1^{\rL}(Q_H,d)$ is the following:

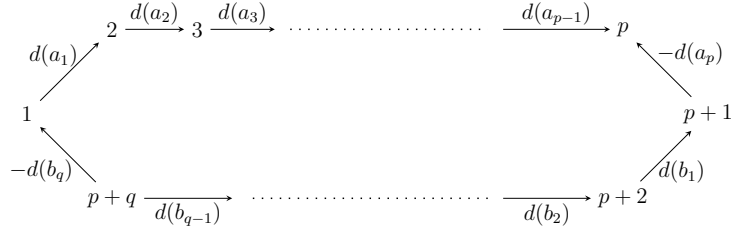
\[\scalebox{.7}{
\begin{tikzpicture}[>=stealth,scale=1.6]
\node (P1) at (0,0) {$1$};
\node (P2) at (1,1) {$2$};
\node (P3) at (2,1) {$3$};
\node (P4) at (3,1){};
\node (P5) at (5.5,1){};
\node (P6) at (7,1) {$p$};
\node (P7) at (8,0) {$p+1$};
\node (P8) at (7,-1) {$p+2$};
\node (P9) at (5.5,-1) {};
\node (P10) at (2.5,-1) {};
\node (P11) at (1,-1) {$p+q$};
\draw [->] (P1)  -- node [swap,xshift=-3mm,yshift=3mm] {$d(a_1)$}(P2);
\draw [->] (P2) -- node [swap,yshift=3mm] {$d(a_2)$}(P3);
\draw [->] (P3) -- node [swap,yshift=3mm] {$d(a_3)$}(P4);
\draw [->] (P5) --node [swap,yshift=3mm] {$d(a_{p-1})$} (P6);
\draw [<-] (P6) -- node [swap,xshift=5mm,yshift=3mm] {$-d(a_p)$}(P7);
\draw [->] (P8) -- node [swap,yshift=-3mm,xshift=3mm] {$d(b_1)$}(P7);
\draw [->] (P9) -- node [swap,yshift=-3mm] {$d(b_2)$}(P8);
\draw [->] (P11) -- node [swap,yshift=-3mm] {$d(b_{q-1})$}(P10);
\draw [<-] (P1) -- node [swap,yshift=-3mm,xshift=-5mm] {$-d(b_q)$}(P11);
\draw [loosely dotted, thick] (P4) -- (P5);
\draw [loosely dotted, thick] (P9) --(P10);
\end{tikzpicture}}
\]
Hence the weight of this grading is $(-d(b_q)+\sum_{i=1}^{p-1}d(a_i))-(\sum_{j=1}^{q-1}d(b_j)-d(a_p))=w(d).$
\end{proof}

 This lemma together with Lemma \ref{gradedeqandw} shows that situation in Remark \ref{counter example} does not occur for a quiver of type $\widetilde{A}_{p,q}$.
Hence we may, in the sequel, refer to the weight $w(\Lambda)$ without specifying a sequence of mutations.

\begin{thma}\label{derivedeqiffwegal}
Let $\Lambda_1$ and $\Lambda_2$ be two algebras of cluster type $\A_{p,q}$ with $p>q$. Then there is a derived equivalence between $\Lambda_1$ and $\Lambda_2$ if and only if $w(\Lambda_1)=w(\Lambda_2)$.
\end{thma}
 \begin{proof}
Let $\delta_1$ (resp.\ $\delta_2$) be a grading on $Q_H$ induced by $\Lambda_1$ (resp.\ $\Lambda_2$).
By Lemma~\ref{gradedeqandw}, $\delta_1$ and $\delta_2$ are equivalent if and only if the corresponding weights $w(\Lambda_1)$ and $w(\Lambda_2)$ coincide. Since, by Lemma~\ref{lem.w_well-def}, the weights are independent of the choice of a sequence of mutation, it follows that also the gradings $\delta_1$ and $\delta_2$ are independent of this choice up to graded equivalence. Now the claim follows from Theorem~\ref{thm_gradingQ_derivedeq}. 
\end{proof}

\begin{cora}
Let $\Lambda$ be an algebra of cluster type $\A_{p,q}$. Then $\Lambda$ is piecewise hereditary if and only if $w(\Lambda)=0$.
\end{cora}
\begin{proof}
Let us treat the case where $p>q$.
The `if'-part of the assertion is a direct consequence of Theorem~\ref{derivedeqiffwegal} applied for $\Lambda_1=\Lambda$ and $\Lambda_2=H$.

The algebra $\Lambda$ is piecewise hereditary if and only if the derived category $\Dd^b(\Lambda)$ is equivalent to $\Dd^b(\Hh)$ for some hereditary category $\Hh$. Therefore it implies that the generalized cluster category $\Cc_\Lambda$ is equivalent to the cluster category $\Cc_\Hh$. Since $\Lambda$ is of cluster type $\A_{p,q}$, we have $\Cc_\Hh\iso \Cc_{\A_{p,q}}$. Hence we get $\Dd^b(\Hh)\iso \Dd^b(\A_{p,q})$. Therefore, by Theorem~\ref{derivedeqiffwegal}, we have $w(\Lambda)=0$. 

The case $p=q$ is a direct consequence of the result below.
\end{proof}

\subsection{The case $p=q$}
In the case $p=q$ there is a unique non-trivial automorphism of $Q_H$ fixing the vertices $1$ and $p+1$ and interchanging the vertices $i$ and $2p+2-i$ for $i=2,\ldots,p$.
This automorphism induces a derived equivalence between algebras $\Lambda_1$ and $\Lambda_2$ of cluster type $\A_{p,p}$ such that $w(\Lambda_1)=-w(\Lambda_2)$. Therefore we obtain the following result, whose proof is the same as Theorem~\ref{derivedeqiffwegal}.

\begin{thma}\label{theorem case p=q}
Let $\Lambda_1$ and $\Lambda_2$ be two algebras of cluster type $\A_{p,p}$. Then there is a derived equivalence between $\Lambda_1$ and $\Lambda_2$ if and only if $|w(\Lambda_1)|=|w(\Lambda_2)|$.
\end{thma}

\section{The Auslander-Reiten quiver of the derived category of an algebra of cluster type $\A_{p,q}$}\label{section ARquiver}

Let $\Lambda$ be an algebra of cluster type $\widetilde{A}_{p,q}$ and of weight $w\neq 0$. In this section we compute the Auslander-Reiten quiver of the derived category $\Dd^b(\Lambda)$. In order to do that, we use some results of \cite[Section 8]{AO10}.
\subsection{Graded derived equivalence}

Let $\Lambda$ be an algebra of cluster type $\A_{p,q}$. Then, by Proposition~\ref{propWgrading}, there exists a graded quiver with reduced potential $(\overline{Q},W,d)$ such that we have an isomorphism of $\ZZ$-graded algebras $\overline{\Lambda}\underset{\ZZ}{\iso}\Jac(\overline{Q},W,d)$. Let $\partial$ be a grading induced by $\Lambda$ on $H$ where $Q_H$ is the following quiver
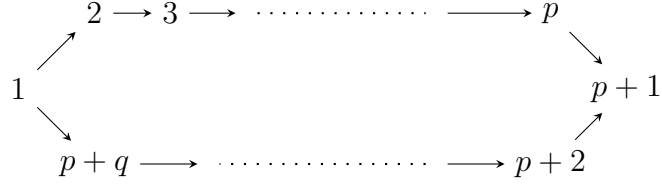
\[\scalebox{1}{
\begin{tikzpicture}[scale=1,>=stealth]
\node (P1) at (0,0) {$1$};
\node (P2) at (1,1) {$2$};
\node (P3) at (2,1) {$3$};
\node (P4) at (3,1){};
\node (P5) at (5.5,1){};
\node (P6) at (7,1) {$p$};
\node (P7) at (8,0) {$p+1$};
\node (P8) at (7,-1) {$p+2$};
\node (P9) at (5.5,-1) {};
\node (P10) at (2.5,-1) {};
\node (P11) at (1,-1) {$p+q$};
\draw [->] (P1) -- (P2);
\draw [->] (P2) -- (P3);
\draw [->] (P3) -- (P4);
\draw [->] (P5) -- (P6);
\draw [->] (P6) -- (P7);
\draw [->] (P8) -- (P7);
\draw [->] (P9) -- (P8);
\draw [->] (P11) -- (P10);
\draw [->] (P1) -- (P11);
\draw [loosely dotted, thick] (P4) -- (P5);
\draw [loosely dotted, thick] (P9) --(P10);
\end{tikzpicture}}.\]
That is, there exists a sequence of mutations such that  $\mu_s(\overline{Q},W,d)=(Q_H,0,\partial)$. Now define the $\ZZ^2$-graded quiver with reduced potential $(Q',W',(d',\delta))$ by the following
\[ (Q',W',(d',\delta)):=\mu_{s^-}^{\rR}(Q_H,0,(\partial,0)).\]
By Lemma~\ref{leftrightmutation}, we have an isomorphism of $\ZZ$-graded algebras $\Jac(Q',W',d')\underset{\ZZ}{\iso}\Jac(\overline{Q},W,d)$. Since the potentials $W$ and $W'$ are reduced, the quivers $\overline{Q}$ and $Q'$ are isomorphic.  By definition of $\overline{Q}$ and $d$, the quiver $Q_\Lambda$ of the algebra $\Lambda$ is the subquiver of $\overline{Q}$ satisfying  $(Q_\Lambda)_0=\overline{Q}_0$ and $(Q_{\Lambda})_1=\{a\in\overline{Q}_1\mid d(a)=0\}$.  Moreover we have an isomorphism $\Lambda\iso kQ_\Lambda/\langle \partial_aW, a\in\overline{Q}_1, d(a)=1\rangle$. Since we have an isomorphism of $\ZZ$-graded algebras $\Jac(Q',W',d')\underset{\ZZ}{\iso}\Jac(\overline{Q},W,d)$, we get an isomorphism $\Lambda\iso kQ_\Lambda/\langle \partial_aW', a\in Q'_1, d'(a)=1\rangle$. The grading $\delta$, which is a grading on $Q'\iso \overline{Q}$, restricts on a grading on $Q_\Lambda$. The grading $\delta$ makes $W'$ homogeneous of degree $1$. Hence the relations $\partial_aW', a\in Q'_1$ with $d'(a)=1$ are homogeneous with respect to the degree $\delta$. Consequently $\delta$ yields a grading on the algebra $\Lambda$ that we still denote by $\delta$.

\medskip
Then we have the following direct consequence of \cite[Theorem~8.7]{AO10}.
\begin{thma}[\cite{AO10}] \label{AO_gradedderivedeq}
In the setup above, there exists a triangle equivalence \[\xymatrix{\Dd^b(\cov{\Lambda}{\delta})\ar[r]^-F_\sim & \Dd^b(\cov{H}{\partial})}\] Moreover we have an isomorphism of triangle functors $F\circ\langle 1\rangle_{\delta}\iso \SSS^{-1}\circ\langle -1\rangle_{\partial}\circ F$.
\end{thma}

\begin{rema}
Note that in this situation, the compatibility condition defined in \cite[Definition~8.5]{AO10} is automatically satisfied since $\overline{Q}$ is mutation acyclic and since $W$ is rigid. Indeed, two $\ZZ$-gradings on a quiver induce a $\ZZ^2$-grading on it. But in general two $\ZZ$-gradings on an algebra do not give rise to a $\ZZ^2$-grading on it. 
\end{rema}

Theorem~\ref{AO_gradedderivedeq} implies the following result which will be very useful to compute explicitly the Auslander-Reiten quiver of $\Dd^b(\Lambda)$.

\begin{cora}\label{k-equivalence}
There exists a $k$-linear equivalence \[\Dd^b( \Lambda)\iso  \Dd^b(\cov{H}{\partial}))/\SSS\langle1\rangle_\partial\]
\end{cora}
 
\begin{proof}
By Theorem~\ref{AO_gradedderivedeq} we deduce that there is a $k$-linear equivalence between the orbit categories \[ \Dd^b(\cov{\Lambda}{\delta})/\langle 1\rangle_\delta\iso \Dd^b(\cov{H}{\partial})/\SSS\langle1\rangle_\partial.\]
Now we can use the following result due to Keller

\begin{thma}[\cite{Kel05}]
Let $Q_H$ be an acyclic quiver  and $\partial$ be a $\ZZ$-grading on $Q_H$.  Let $F:=-\lten_{\cov{H}{\partial}} X$ be an auto-equivalence of $\Dd^b(\cov{H}{\partial})$ for some object $X\in\Dd(\cov{H}{\partial})$. Assume that \begin{enumerate}
\item For each indecomposable $U\in\mod\cov{H}{\partial}$, the set $\{i\in\ZZ\mid F^iU\in\mod\cov{H}{\partial}\}$ is finite.
\item There exists $N\geq 0$ such that for each indecomposable $X\in\Dd^b(\cov{H}{\partial})$, there exists $0\leq n\leq N$ and $i\in\ZZ$ with $F^iX[-n]\in\mod\cov{H}{\partial}$.
\end{enumerate}
Then the orbit category $\Dd^b(\cov{H}{\partial})/F$ admits a natural triangulated structure such that the projection functor $\Dd^b(\cov{H}{\partial})\rightarrow\Dd^b(\cov{H}{\partial})/F$ is triangulated.
\end{thma}

It is already shown in \cite{BMRRT} that the functor $\SSS$ satisfies  conditions (1) and (2). Since the functor $\langle 1\rangle_\partial$ is an auto-equivalence of $\mod\cov{H}{\partial}$, then the functor $\SSS\langle 1\rangle_{\partial}$ clearly satisfies again conditions (1) and (2). Thus the orbit category $(\Dd^b(\cov {\Lambda}{\delta})/\langle 1\rangle_\delta)\iso \Dd^b(\cov{H}{\partial})/\SSS\langle1\rangle_\partial$ is triangulated and the natural functor \[ \xymatrix{\Dd^b(\cov{\Lambda}{\delta})\ar[r] & \Dd^b(\cov{\Lambda}{\delta})/\langle 1\rangle_\delta}\] is a triangle functor. 

By \cite[Corollary~7.14]{AO10}, the derived category $\Dd^b(\Lambda)$ is equivalent to the triangulated hull $(\Dd^b(\cov{\Lambda}{\delta})/\langle 1\rangle_\delta)_\Delta $. Hence we have the following commutative diagram \[\xymatrix{\Dd^b(\cov{\Lambda}{\delta})\ar[r]^-A\ar@/_.5cm/[rr]_-C & \Dd^b(\cov{\Lambda}{\delta})/\langle 1 \rangle\ar@{^(->}[r]^-B & \Dd^b(\Lambda)}\] where $A$ and $C$ are triangle functors. Therefore $B$ commutes with the shift and sends triangles to triangles. Hence the fully faithful functor $B$ is a $k$-equivalence \[\Dd^b(\cov{\Lambda}{\delta}/\langle 1\rangle_\delta)\iso \Dd^b(\cov{H}{\partial})/\SSS\langle1\rangle_\partial. \qedhere\]
\end{proof}

\subsection{The shape of the Auslander-Reiten quiver of $\Dd^b(\Lambda)$}

In this subsection we use Corollary~\ref{k-equivalence} to compute explicitly the shape of the Auslander-Reiten quiver of the derived category of an algebra of cluster type $\A_{p,q}$. 

Throughout this subsection $\Lambda$ is an algebra of cluster type $\A_{p,q}$ and of weight $w\neq 0$. Let $\partial$ be a grading induced by $\Lambda$ on $Q_H$. By Lemma~\ref{gradedeqandw} we can assume that $\partial$ is the following grading.
 \[\scalebox{.7}
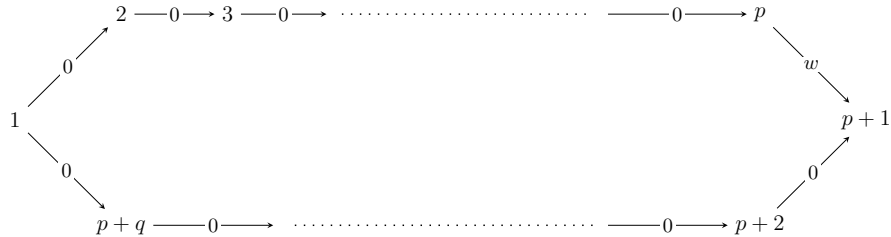
{
\begin{tikzpicture}[scale=2,>=stealth]
\node (P1) at (0,0) {$1$};
\node (P2) at (1,1) {$2$};
\node (P3) at (2,1) {$3$};
\node (P4) at (3,1){};
\node (P5) at (5.5,1){};
\node (P6) at (7,1) {$p$};
\node (P7) at (8,0) {$p+1$};
\node (P8) at (7,-1) {$p+2$};
\node (P9) at (5.5,-1) {};
\node (P10) at (2.5,-1) {};
\node (P11) at (1,-1) {$p+q$};
\draw [->] (P1) -- node [fill=white,inner sep=.5mm]{\small{0}} (P2);
\draw [->] (P2) -- node [fill=white,inner sep=.5mm]{\small{0}} (P3);
\draw [->] (P3) -- node [fill=white,inner sep=.5mm]{\small{0}} (P4);
\draw [->] (P5) --node [fill=white,inner sep=.5mm]{\small{0}}  (P6);
\draw [->] (P6) --node [fill=white,inner sep=.5mm]{\small{$w$}}  (P7);
\draw [->] (P8) --node [fill=white,inner sep=.5mm]{\small{0}}  (P7);
\draw [->] (P9) --node [fill=white,inner sep=.5mm]{\small{0}}  (P8);
\draw [->] (P11) -- node [fill=white,inner sep=.5mm]{\small{0}} (P10);
\draw [->] (P1) -- node [fill=white,inner sep=.5mm]{\small{0}} (P11);
\draw [loosely dotted, thick] (P4) -- (P5);
\draw [loosely dotted, thick] (P9) --(P10);
\end{tikzpicture}}
\]
By Corollary~\ref{k-equivalence} we have a $k$-linear equivalence \[\Dd^b(\Lambda)\iso \Dd^b(\cov{H}{\partial})/\SSS\langle1\rangle_\partial.\]

The quiver of the subcategory $\cov{H}{\partial}=\add\{H\langle p\rangle_{\partial}, p\in\mathbb{Z}\}\subset \gr H$ consists of $|w|$ connected components of type $A_\infty^\infty$ with the following orientation.
 \[\scalebox{.5}
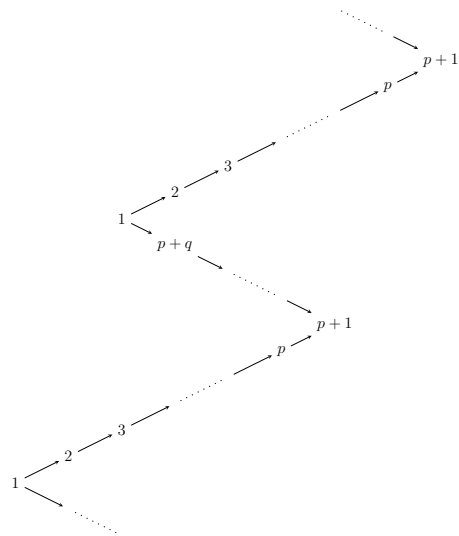
{
\begin{tikzpicture}[scale=.7,>=stealth]
\node (P1) at (0,0) {$1$};
\node (P2) at (2,1) {$2$};
\node (P3) at (4,2) {$3$};
\node (P4) at (6,3) {};
\node (P5) at (8,4) {};
\node (P6) at (10,5) {$p$};
\node (P7) at (12,6) {$p+1$};
\node (P8) at (10,7) {}; \node (P8') at (8,8) {};
\node (P9) at (2,-1) {$p+q$};
\node (P10) at (4,-2) {};
\node (P11) at (6,-3) {};
\node (P12) at (8,-4) {$p+1$};
\node (Q1) at (-4,-10) {$1$};
\node (Q2) at (-2,-9) {$2$};
\node (Q3) at (0,-8) {$3$};
\node (Q4) at (2,-7) {};
\node (Q5) at (4,-6) {};
\node (Q6) at (6,-5) {$p$};
\node (Q7) at (-2,-11) {};\node (Q7') at (0,-12) {};
\draw [->] (P1)--(P2);
\draw [->] (P2)--(P3);
\draw [->] (P3)--(P4);
\draw [loosely dotted, thick] (P4)--(P5);
\draw [->] (P5)--(P6);
\draw [->] (P6)--(P7);
\draw [->] (P8)--(P7);
\draw [->] (P1)--(P9);
\draw [->] (P9)--(P10);
\draw [loosely dotted, thick] (P10)--(P11);
\draw [->] (P11)--(P12);
\draw [->] (Q6)--(P12);
\draw [->] (Q1)--(Q2);
\draw [->] (Q2)--(Q3);
\draw [->] (Q3)--(Q4);
\draw [loosely dotted, thick] (Q4)--(Q5);
\draw [->] (Q5)--(Q6);
\draw [->] (Q1)--(Q7);
\draw [loosely dotted, thick](Q7')--(Q7);
\draw [loosely dotted, thick] (P8)--(P8');
\end{tikzpicture}}
\]
Then it is not hard  to compute the Auslander-Reiten quiver of the module category $\mod\cov{H}{\partial}$. It has $4|w|$ connected components:
 $\mathcal{P}_{i}$,  $\mathcal{R}^p_{i}$ , $\mathcal{R}^q_{i}$, and  $\mathcal{I}_i$ with  $i\in\ZZ/w\ZZ$
satisfying  $\mathcal{P}_{i}=\mathcal{P}_{0}\langle i\rangle$,  $\mathcal{R}^p_{i}=\mathcal{R}^p_{0}\langle i\rangle$,  
$\mathcal{R}^q_{i}=\mathcal{R}^q_{0}\langle i\rangle$ and $\mathcal{I}_i=\mathcal{I}_0\langle i \rangle$ for any $i\in\mathbb{Z}/w\mathbb{Z}.$

The component $\Pp_0$  is of type $\mathbb{N}A_\infty^\infty$ and contains the indecomposable projective modules  $P_j\langle 0 \rangle$ for $j=1,\ldots, p+q$. Its shape is described in Figure~\ref{figureP_0}.
\begin{figure}

\[ \scalebox{.5}{
\begin{tikzpicture}
\node (P1+0) at (0,0) {$P_1\langle 0 \rangle$};
 \foreach \i in {1,...,4} \node (P1+\i) at (4 * \i, 0) {};
\node (P2+0) at (2,1) {$P_2\langle 0 \rangle$};
 \foreach \i in {1,2,3} \node (P2+\i) at (2 + 4 * \i, 1) {};
\node (P3+0) at (4,2) {$P_3\langle 0 \rangle$};
 \foreach \i in {1,2,3} \node (P3+\i) at (4 + 4 * \i, 2) {};
\node (P4+0) at (6,3) {};
 \foreach \i in {1,2} \node (P4+\i) at (6 + 4 * \i, 3) {};
\node (P5+0) at (8,4) {};
 \foreach \i in {1,2} \node (P5+\i) at (8 + 4 * \i, 4) {};
\node (P6+0) at (10,5) {$P_p\langle 0 \rangle$};
 \node (P6+1) at (14, 5) {};
\node (P7+0) at (12,6) {$P_{p+1}\langle w\rangle$};
 \node (P7+1) at (16, 6) {};
\node (P8+0) at (10,7) {};
 \node (P8+1) at (14, 7) {};
\node (P9+0) at (2,-1) {$P_{p+q}\langle 0 \rangle$};
  \foreach \i in {1,2,3} \node (P9+\i) at (2 + 4 * \i, -1) {};
\node (P10+0) at (4,-2) {};
  \foreach \i in {1,2,3} \node (P10+\i) at (4 + 4 * \i, -2) {};
\node (P11+0) at (6,-3) {};
 \foreach \i in {1,2} \node (P11+\i) at (6 + 4 * \i, -3) {};
\node (P12+0) at (8,-4) {$P_{p+1}\langle 0 \rangle$};
 \foreach \i in {1,2} \node (P12+\i) at (8 + 4 * \i, -4) {};
\node (Q6+0) at (6,-5) {$P_p\langle -w\rangle$};
 \foreach \i in {1,2} \node (Q6+\i) at (6 + 4 * \i, -5) {};

\foreach \i in {0,1} \draw [ultra thick, loosely dotted] (9 + 4 * \i, 7.5) -- (P8+\i);
\foreach \i in {0,1} \draw [ultra thick, loosely dotted] (P8+\i) -- (11 + 4 * \i, 7.5);

\foreach \i in {0,1} \draw [->] (P8+\i) -- (P7+\i);
\draw [->] (P7+0) -- (P8+1); \draw [ultra thick, loosely dotted] (P7+1) -- (17,6.5);

\foreach \i in {0,1} \draw [->] (P6+\i) -- (P7+\i);
\draw [->] (P7+0) -- (P6+1); \draw [ultra thick, loosely dotted] (P7+1) -- (17,5.5);

\foreach \i in {0,1} \draw [->] (P5+\i) -- (P6+\i);
\foreach \i/\j in {0/1,1/2} \draw [->] (P6+\i) -- (P5+\j);
\draw [ultra thick, loosely dotted] (P5+2) -- (17,4.5);

\foreach \i in {0,1,2} \draw [ultra thick, loosely dotted] (P4+\i) -- (P5+\i);

\foreach \i in {0,1,2} \draw [->] (P3+\i) -- (P4+\i);
\foreach \i/\j in {0/1,1/2,2/3} \draw [->] (P4+\i) -- (P3+\j);
\draw [ultra thick, loosely dotted] (P3+3) -- (17,2.5);

\foreach \i in {0,...,3} \draw [->] (P2+\i) -- (P3+\i);
\foreach \i/\j in {0/1,1/2,2/3} \draw [->] (P3+\i) -- (P2+\j);
\draw [ultra thick, loosely dotted] (P3+3) -- (17,1.5);

\foreach \i in {0,...,3} \draw [->] (P1+\i) -- (P2+\i);
\foreach \i/\j in {0/1,1/2,2/3,3/4} \draw [->] (P2+\i) -- (P1+\j);
\draw [ultra thick, loosely dotted] (P1+4) -- (17,.5);

\foreach \i in {0,...,3} \draw [->] (P1+\i) -- (P9+\i);
\foreach \i/\j in {0/1,1/2,2/3,3/4} \draw [->] (P9+\i) -- (P1+\j);
\draw [ultra thick, loosely dotted] (P1+4) -- (17,-.5);

\foreach \i in {0,...,3} \draw [->] (P9+\i) -- (P10+\i);
\foreach \i/\j in {0/1,1/2,2/3} \draw [->] (P10+\i) -- (P9+\j);
\draw [ultra thick, loosely dotted] (P10+3) -- (17,-1.5);

\foreach \i in {0,1,2} \draw [ultra thick, loosely dotted] (P10+\i) -- (P11+\i);
\draw [ultra thick, loosely dotted] (P10+3) -- (17,-2.5);

\foreach \i in {0,1,2} \draw [->] (P11+\i) -- (P12+\i);
\foreach \i/\j in {0/1,1/2} \draw [->] (P12+\i) -- (P11+\j);
\draw [ultra thick, loosely dotted] (P12+2) -- (17,-3.5);

\foreach \i in {0,1,2} \draw [->] (Q6+\i) -- (P12+\i);
\foreach \i/\j in {0/1,1/2} \draw [->] (P12+\i) -- (Q6+\j);
\draw [ultra thick, loosely dotted] (P12+2) -- (17,-4.5);

\foreach \i in {0,1,2} \draw [ultra thick, loosely dotted] (5 + 4 * \i, -5.5) -- (Q6+\i);
\foreach \i in {0,1,2} \draw [ultra thick, loosely dotted] (Q6+\i) -- (7 + 4 * \i, -5.5);
\end{tikzpicture} } \]
\caption{\label{figureP_0}Shape of the component $\Pp_{0}$ of the Auslander-Reiten quiver of $\mod(\cov{H}{\partial})$. } \end{figure}

The component $\Ii_0$ is of type $(-\mathbb{N})A^\infty_\infty$ and contains the indecomposable injective modules $I_j\langle 0\rangle$ for $j=1,\ldots,p+q$. Its shape is described in Figure~\ref{figureI_0}.
\begin{figure}
\[
\scalebox{.5}{
\begin{tikzpicture}[scale=.9,>=stealth]
\node (P0-0) at (4,0) {$I_1\langle 0 \rangle$};
\node (P0-1) at (0,0){};
\node (P1-0) at (6,1) {$I_2\langle 0 \rangle$};
\node (P1-1) at (2,1){};
\node (P2-0) at (8,2) {$I_3\langle 0 \rangle$};
 \foreach \i in {1,2} \node (P2-\i) at (8-4 * \i, 2) {};
 \foreach \i in {0,1,2} \node (P3-\i) at (10-4 * \i, 3) {};
 \foreach \i in {0,...,3} \node (P4-\i) at (12-4 * \i, 4) {};
\node (P5-0) at (14,5) {$I_p\langle 0 \rangle$};
 \foreach \i in {1,2,3} \node (P5-\i) at (14-4 * \i, 5) {};
\node (P6-0) at (16,6) {$I_{p+1}\langle w\rangle$};
 \foreach \i in {1,...,4} \node (P6-\i) at (16-4 * \i, 6) {};
  \foreach \i in {0,...,3} \node (P7-\i) at (14-4 * \i, 7) {};
\node (P-1-0) at (6,-1) {$I_{p+q}\langle 0 \rangle$};
\node (P-1-1) at (2,-1){};
\foreach \i in {0,1,2} \node (P-2-\i) at (8-4*\i,-2) {};
\foreach \i in {0,1,2} \node (P-3-\i) at (10-4*\i,-3) {};
\node (P-4-0) at (12,-4) {$I_{p+1}\langle 0 \rangle$};
\foreach\i in {1,2,3} \node (P-4-\i) at (12-4*\i,-4) {};
\node (P-5-0) at (10,-5) {$I_p\langle -w\rangle$};
\foreach \i in {1,2} \node (P-5-\i) at (10-4*\i,-5) {};

\foreach \i in {0,1} \draw[->] (P0-\i)--(P1-\i);
\draw[->](P1-1)--(P0-0);
\foreach \i in {0,1} \draw[->] (P1-\i)--(P2-\i);
\foreach \i/\j in {2/1,1/0} \draw [->] (P2-\i) -- (P1-\j);
\foreach \i in {0,1,2} \draw[->] (P2-\i)--(P3-\i);
\foreach \i/\j in {2/1,1/0} \draw [->] (P3-\i) -- (P2-\j);
\foreach \i in {0,1,2} \draw[ultra thick, loosely dotted] (P3-\i)--(P4-\i);
\foreach \i in {0,1,2,3} \draw[->] (P4-\i)--(P5-\i);
\foreach \i/\j in {3/2,2/1,1/0} \draw[->](P5-\i)--(P4-\j);
\foreach \i in {0,...,3} \draw[->] (P5-\i)--(P6-\i);
\foreach \i/\j in {4/3,3/2,2/1,1/0} \draw [->] (P6-\i) -- (P5-\j);
\foreach \i in {0,...,3} \draw[<-] (P6-\i)--(P7-\i);
\foreach \i/\j in {4/3,3/2,2/1,1/0} \draw [<-] (P7-\j) -- (P6-\i);

\foreach \i in {0,1} \draw [->] (P0-\i)--(P-1-\i);
\draw[->] (P-1-1)--(P0-0);
\foreach \i/\j in {1/0,2/1} \draw[->] (P-2-\i)--(P-1-\j);
\foreach \i in {0,1} \draw [->] (P-1-\i)--(P-2-\i);
\foreach \i in {0,1,2} \draw[ultra thick, loosely dotted] (P-2-\i)--(P-3-\i);
\foreach \i in {0,1,2} \draw[->] (P-3-\i)--(P-4-\i);
\foreach \i/\j in {1/0,2/1,3/2} \draw[->] (P-4-\i)--(P-3-\j);
\foreach \i/\j in {1/0,2/1,3/2} \draw[->] (P-4-\i)--(P-5-\j);
\foreach \i in {0,1,2} \draw[<-] (P-4-\i)--(P-5-\i);

\foreach \i in {1,...,3} \draw[ultra thick, loosely dotted] (P7-\i)--(15-4*\i,7.5);
\foreach \i in {0,...,3} \draw[ultra thick, loosely dotted] (P7-\i)--(13-4*\i,7.5);
\foreach \i in {1,2} \draw[ultra thick, loosely dotted] (P-5-\i)--(11-4*\i,-5.5);
\foreach \i in {0,...,2} \draw[ultra thick, loosely dotted] (P-5-\i)--(9-4*\i,-5.5);

\foreach \i in {-4,0,2,4,6} \draw[ultra thick, loosely dotted] (0,\i)--(-1,\i-0.5);
\foreach \i in {-4,-2,0,2,4,6} \draw[ultra thick, loosely dotted] (0,\i)--(-1,\i+0.5);

\end{tikzpicture}}\]
\caption{\label{figureI_0}Shape of the component $\mathcal{I}_{0}$ of the Auslander-Reiten quiver of $\mod(\cov{H}{\partial})$. }

\end{figure}
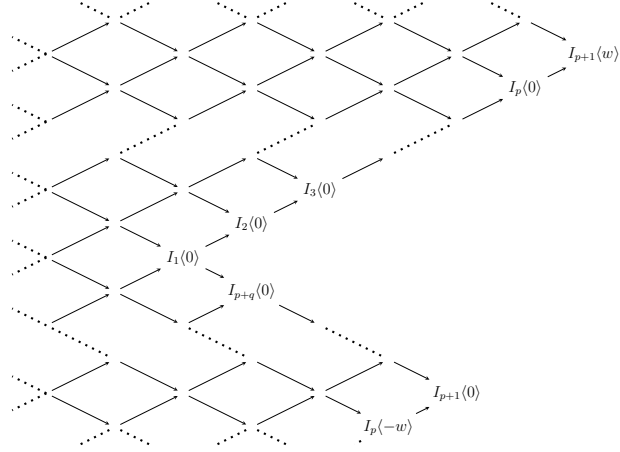

The components $\Rr^p_0$ and $\Rr^q_0$ are of type $\mathbb{Z}A_\infty$ and contain the regular modules. The shape of the connected components $\mathcal{R}^p_{0}$ and $\mathcal{R}_{0}^q$ are described in Figures~\ref{figureRp_0} and~\ref{figureRq_0}. In these figures, $S_i\langle m\rangle$ is the $m$-degree-shift of the simple module associated with the vertex $i$, the object $M\langle 0 \rangle$ is defined by the exact sequence \[\xymatrix{P_{p}\langle -w\rangle\ar[r]^{a_p} & P_{p+1}\langle 0 \rangle\ar[r]& M\langle 0 \rangle\ar[r] & 0},\] and the object $N\langle 0 \rangle $ is defined by the exact sequence   
\[\xymatrix{P_{p+2}\langle 0\rangle\ar[r]^{b_{q}} & P_{p+1}\langle 0 \rangle\ar[r]& N\langle 0 \rangle\ar[r] & 0}.\] 

\begin{figure}
\[\scalebox{.5}{
\begin{tikzpicture}[scale=1,>=stealth]

\node (P1) at (0,0) {$S_2\langle 0\rangle$};
\node (P2) at (2,0) {$S_3\langle 0 \rangle$};
\node (P3) at (4,0) {};
\node (P4) at (8,0) {};
\node (P5) at (10,0) {$S_p\langle 0\rangle$};
\node (P6) at (12,0) {$M\langle w\rangle$};
\node (P7) at (14,0) {$S_2\langle w\rangle$};

\node (Q1) at (0,2) {};
\node (Q2) at (2,2) {};
\node (Q3) at (4,2) {};
\node (Q4) at (8,2) {};
\node (Q5) at (10,2) {};
\node (Q6) at (12,2) {};
\node (Q7) at (14,2) {};

\node (R1) at (-1,1) {};
\node (R2) at (1,1) {};
\node (R3) at (3,1) {};
\node (R4) at (9,1) {};
\node (R5) at (11,1) {};
\node (R6) at (13,1) {};
\node (R7) at (15,1) {};

\draw [->] (R1)--(Q1);
\draw [->] (R1)--(P1);
\draw [->] (R2)--(Q2);
\draw [->] (R2)--(P2);
\draw [->] (R3)--(Q3);
\draw [->] (R3)--(P3);
\draw [->] (R4)--(Q5);
\draw [->] (R4)--(P5);
\draw [->] (R5)--(Q6);
\draw [->] (R5)--(P6);
\draw [->] (R6)--(P7);
\draw [->] (R6)--(Q7);
\draw [->] (P1)--(R2);
\draw [->] (Q1)--(R2);
\draw [->] (P2)--(R3);
\draw [->] (Q2)--(R3);
\draw [->] (P4)--(R4);
\draw [->] (P5)--(R5);
\draw [->] (P6)--(R6);
\draw [->] (P7)--(R7);
\draw [->] (Q4)--(R4);
\draw [->] (Q5)--(R5);
\draw [->] (Q6)--(R6);
\draw [->] (Q7)--(R7);

\draw [loosely dotted, ultra thick] (P3)--(P4);
\draw [loosely dotted, ultra thick] (Q3)--(Q4);
\draw [loosely dotted, ultra thick] (-3,1)--(R1);
\draw [loosely dotted, ultra thick] (0,4)--(Q1);
\draw [loosely dotted, ultra thick] (17,1)--(R7);
\draw [loosely dotted, ultra thick] (14,4)--(Q7);
\end{tikzpicture}}
\]\caption{\label{figureRp_0} Shape of the connected component $\mathcal{R}^p_{0}$.}
 \end{figure}

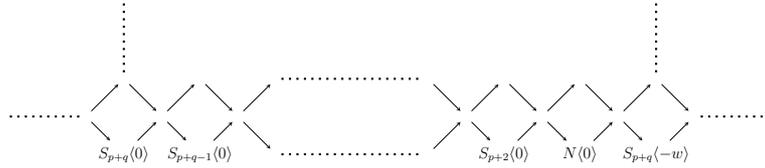
\begin{figure}
\[\scalebox{.5}{
\begin{tikzpicture}[scale=1,>=stealth]

\node (P1) at (0,0) {$S_{p+q}\langle 0\rangle$};
\node (P2) at (2,0) {$S_{p+q-1}\langle 0 \rangle$};
\node (P3) at (4,0) {};
\node (P4) at (8,0) {};
\node (P5) at (10,0) {$S_{p+2}\langle 0\rangle$};
\node (P6) at (12,0) {$N\langle 0 \rangle$};
\node (P7) at (14,0) {$S_{p+q}\langle -w\rangle$};

\node (Q1) at (0,2) {};
\node (Q2) at (2,2) {};
\node (Q3) at (4,2) {};
\node (Q4) at (8,2) {};
\node (Q5) at (10,2) {};
\node (Q6) at (12,2) {};
\node (Q7) at (14,2) {};

\node (R1) at (-1,1) {};
\node (R2) at (1,1) {};
\node (R3) at (3,1) {};
\node (R4) at (9,1) {};
\node (R5) at (11,1) {};
\node (R6) at (13,1) {};
\node (R7) at (15,1) {};

\draw [->] (R1)--(Q1);
\draw [->] (R1)--(P1);
\draw [->] (R2)--(Q2);
\draw [->] (R2)--(P2);
\draw [->] (R3)--(Q3);
\draw [->] (R3)--(P3);
\draw [->] (R4)--(Q5);
\draw [->] (R4)--(P5);
\draw [->] (R5)--(Q6);
\draw [->] (R5)--(P6);
\draw [->] (R6)--(P7);
\draw [->] (R6)--(Q7);
\draw [->] (P1)--(R2);
\draw [->] (Q1)--(R2);
\draw [->] (P2)--(R3);
\draw [->] (Q2)--(R3);
\draw [->] (P4)--(R4);
\draw [->] (P5)--(R5);
\draw [->] (P6)--(R6);
\draw [->] (P7)--(R7);
\draw [->] (Q4)--(R4);
\draw [->] (Q5)--(R5);
\draw [->] (Q6)--(R6);
\draw [->] (Q7)--(R7);

\draw [loosely dotted, ultra thick] (P3)--(P4);
\draw [loosely dotted, ultra thick] (Q3)--(Q4);
\draw [loosely dotted, ultra thick] (-3,1)--(R1);
\draw [loosely dotted, ultra thick] (0,4)--(Q1);
\draw [loosely dotted, ultra thick] (17,1)--(R7);
\draw [loosely dotted, ultra thick] (14,4)--(Q7);
\end{tikzpicture}}
\]\caption{\label{figureRq_0}Shape of the connected component $\mathcal{R}^q_{0}$.}
 \end{figure}
Under the forgetful functor $\mod\cov{H}{\partial}\rightarrow\mod H$, the components $\Pp_i$ are sent to $\Pp$ the preprojective component of $\mod H$, the components $\Ii_i$ are sent to the preinjective component $\Ii$, the components $\mathcal{R}^p_j$ are sent to the exceptional tube of rank $p$ and the components $\mathcal{R}^q_j$ are sent to the exceptional tube of rank $q$. The indecomposable $H$-modules lying in the homogeneous tubes are not gradable if $w\neq 0$, therefore the description above is complete.

\medskip

Since $\mod\cov{H}{\partial}$ is a hereditary category, one can easily deduce the shape of the Auslander-Reiten quiver of $\Dd^b(\cov{H}{\partial})$. It has three kinds of components 
$\mathcal{P}_{(i,n)}$,  $\mathcal{R}^p_{(i,n)}$, and $\mathcal{R}^q_{(i,n)}$, with  $i\in\ZZ/w\ZZ$ and $n\in\ZZ$. The component $\Pp_{(0,0)}$ contains $\Pp_0$ and $\Ii_0[-1]$ and is described in Figure~\ref{figureP_00}. The component $\mathcal{R}^p_{(0,0)}$ (resp.\ $\mathcal{R}^q_{(0,0)}$) is $\mathcal{R}^p_{0}$ (resp.\ $\mathcal{R}^q_{0}$). Moreover we have $\mathcal{P}_{(i,n)}=\mathcal{P}_{(0,0)}\langle i\rangle [n]$, $\mathcal{R}^p_{(i,n)}=\mathcal{R}^p_{(0,0)}\langle i \rangle [n]$, and $\mathcal{R}^q_{(i,n)}=\mathcal{R}^q_{(0,0)}\langle i\rangle [n]$ for $i\in\mathbb{Z}/w\ZZ$ and $n\in\ZZ$.

 \begin{figure}
\[
\scalebox{.5}{
\begin{tikzpicture}[scale=.9,>=stealth]
\foreach \i in {-5,-3,-1,1,3,5,7} \node (-3-\i) at (-6,\i){};
\node (-2-0) at (-4,0) {$I_1\langle 0 \rangle[-1]$};
\foreach \i in {-4,-2,2,4,6} \node (-2-\i) at (-4,\i){};
\node (-1-1) at (-2,1) {$I_2\langle 0 \rangle[-1]$};\node (-1--1) at (-2,-1) {$I_{p+q}\langle 0 \rangle[-1]$};
\foreach \i in {-5,-3,3,5,7} \node (-1-\i) at (-2,\i){};
\node (0-0) at (0,0) {$P_1\langle 0 \rangle$};\node (0-2) at (0,2) {$I_3\langle 0 \rangle[-1]$};
\foreach \i in {-4,-2,4,6} \node (0-\i) at (0,\i){};
\node (1-1) at (2,1) {$P_2\langle 0 \rangle$};\node (1--5) at (2,-5) {$I_p\langle -w\rangle[-1]$};\node (1--1) at (2,-1) {$P_{p+q}\langle 0 \rangle$};
\foreach \i in {-3,3,5,7} \node (1-\i) at (2,\i){};
\node (2-2) at (4,2) {$P_3\langle 0 \rangle$};\node (2--4) at (4,-4) {$I_{p+1}\langle 0 \rangle[-1]$};
\foreach \i in {-2,0,4,6} \node (2-\i) at (4,\i){};
\node (3-5) at (6,5) {$I_p\langle 0 \rangle[-1]$};\node (3--5) at (6,-5) {$P_p\langle -w\rangle$};
\foreach \i in {-3,-1,1,3,7} \node (3-\i) at (6,\i){};
\node (4-6) at (8,6) {$I_{p+1}\langle w\rangle[-1]$};\node (4--4) at (8,-4) {$P_{p+1}\langle 0 \rangle$};
\foreach \i in {-2,0,2,4} \node (4-\i) at (8,\i){};
\node (5-5) at (10,5) {$P_p\langle 0 \rangle$};
\foreach \i in {-5,-3,-1,1,3,7} \node (5-\i) at (10,\i){};
\node (6-6) at (12,6) {$P_{p+1}\langle w\rangle$};
\foreach \i in {-4,-2,0,2,4} \node (6-\i) at (12,\i){};
\foreach \i in {-5,-3,-1,1,3,5,7} \node (7-\i) at (14,\i){};
\foreach \i in {-5,-3,-1,1,3,5,7} \node (\i--7) at (2*\i,-7){};
\foreach \i in {-4,-2,0,2,4,6}\node (\i--6) at (2*\i, -6){};

\foreach \k/\l in {-2/-1,0/1,2/3,4/5,6/7} \foreach \i/\j in {-6/-5,-4/-3,-2/-1,0/1,2/3,4/5,6/7} \draw[->] (\k-\i)--(\l-\j);
\foreach \k/\l in {-2/-1,0/1,2/3,4/5,6/7} \foreach \i/\j in {-6/-7,-4/-5,0/-1,2/1,6/5} \draw[->] (\k-\i)--(\l-\j);
\foreach \k/\l in {-2/-3,0/-1,2/1,4/3,6/5} \foreach \i/\j in {-6/-5,-4/-3,-2/-1,0/1,2/3,4/5,6/7} \draw[<-] (\k-\i)--(\l-\j);
\foreach \k/\l in {-2/-3,0/-1,2/1,4/3,6/5} \foreach \i/\j in {-6/-7,-4/-5,0/-1,2/1, 6/5} \draw[<-] (\k-\i)--(\l-\j);

\foreach \k/\l in {-2/-3,0/-1,2/1,4/3,6/5} \draw[ultra thick, loosely dotted] (\k-4)--(\l-3);
\foreach \k/\l in {-2/-1,0/1,2/3,4/5,6/7}\draw[ultra thick, loosely dotted] (\k--2)--(\l--3);

\foreach \i in {-7,-5,-1,1,3,5,7} \draw[ultra thick, loosely dotted] (14,\i)--(15,\i+0.5);
\foreach \i in {-7,-5,-3,-1,1,3,5,7} \draw[ultra thick, loosely dotted] (14,\i)--(15,\i-0.5);
\foreach \i in {-7,-5,-3,-1,1,5,7} \draw[ultra thick, loosely dotted] (-6,\i)--(-7,\i+0.5);
\foreach \i in {-7,-5,-3,-1,1,3,5,7} \draw[ultra thick, loosely dotted] (-6,\i)--(-7,\i-0.5);
\foreach \i in {-3,-1,1,3,5,7} \draw[ultra thick, loosely dotted] (2*\i,7)--(2*\i+1,7.5);
 \foreach \i in {-3,-1,1,3,5,7} \draw[ultra thick, loosely dotted] (2*\i,7)--(2*\i-1,7.5);
 \foreach \i in {-3,-1,1,3,5,7} \draw[ultra thick, loosely dotted] (2*\i,-7)--(2*\i+1,-7.5);
 \foreach \i in {-3,-1,1,3,5,7} \draw[ultra thick, loosely dotted] (2*\i,-7)--(2*\i-1,-7.5);
\end{tikzpicture}}\]
\caption{\label{figureP_00}Shape of the component $\Pp_{(0,0)}$ of the Auslander-Reiten quiver of $\Dd^b(\cov{H}{\partial})$. }\end{figure}

 The morphisms satisfy the following:
\[\begin{array}{lll}\Hom(\Pp_{(i,n)},\Pp_{(j,m)})\neq 0&\textrm{ if and only if} & i=j \textrm{ and } m \in \{n,n+1\} \\
\Hom(\Pp_{(i,n)},\Rr^{p,q}_{(j,m)})\neq 0 &\textrm{ if and only if} & i=j \textrm{ and } m=n\\
\Hom(\Rr^{p,q}_{(i,n)},\Pp_{(j,m)})\neq 0 &\textrm{ if and only if} & i=j \textrm{ and } m=n+1\\
\Hom(\Rr^{p,q}_{(i,n)},\Rr^{p,q}_{(j,m)})\neq 0 &\textrm{ if and only if} & i=j \textrm{ and } m \in \{n,n+1\}
\end{array}.\]

\begin{figure}
\[\scalebox{1}{\begin{tikzpicture}[scale=.3,>=stealth]
\draw [loosely dotted, thick] (0,0) -- (5,0) -- (5,6) -- (0,6) -- cycle;
\draw [loosely dotted, thick] (0,5)--(-1,5)--(-1,-1)--(4,-1)--(4,0);
\draw [loosely dotted, thick] (-1,3)--(-3,3)--(-3,-3)--(2,-3)--(2,-1);
\node (P0) at (0,-5) {$\mathcal{P}_{(i,0)}$};
\draw [loosely dotted, thick] (15,0)--(15,6)--(10,6)--(10,0);
\draw (10,0)--(15,0);
\draw [loosely dotted, thick] (10,5)--(9,5)--(9,-1);
\draw [loosely dotted, thick] (14,-1)--(14,0);
\draw (9,-1)--(14,-1);
\draw [loosely dotted, thick] (9,3)--(7,3)--(7,-3);
\draw [loosely dotted, thick] (12,-3)--(12,-1);
\draw (7,-3)--(12,-3);
\node (P2) at (10,-5) {$\mathcal{R}_{(i,0)}^{p}$};
\draw [loosely dotted, thick] (24,0)--(24,6)--(19,6)--(19,0);
\draw (19,0)--(24,0);
\draw [loosely dotted, thick] (19,5)--(18,5)--(18,-1);
\draw [loosely dotted, thick] (23,-1)--(23,0);
\draw (18,-1)--(23,-1);
\draw [loosely dotted, thick] (18,3)--(16,3)--(16,-3);
\draw [loosely dotted, thick] (21,-3)--(21,-1);
\draw (16,-3)--(21,-3);
\node (P1) at (19,-5) {$\mathcal{R}_{(i,0)}^q$};

\draw [loosely dotted, thick] (29,0) --(34,0)--(34,6)--(29,6)--(29,0);
\draw [loosely dotted, thick] (29,5)--(28,5)--(28,-1)--(33,-1)--(33,0);
\draw [loosely dotted, thick] (28,3)--(26,3)--(26,-3)--(31,-3)--(31,-1);
\node (P4) at (29,-5) {$\mathcal{P}_{(i,1)}$};

\draw [loosely dotted, thick] (44,0)--(44,6)--(39,6)--(39,0);
\draw (39,0)--(44,0);
\draw [loosely dotted, thick] (39,5)--(38,5)--(38,-1);
\draw [loosely dotted, thick] (43,-1)--(43,0);
\draw (38,-1)--(43,-1);
\draw [loosely dotted, thick] (38,3)--(36,3)--(36,-3);
\draw [loosely dotted, thick] (41,-3)--(41,-1);
\draw (36,-3)--(41,-3);
\node at (39,-5) {$\mathcal{R}_{(i,1)}^p$};

\draw [loosely dotted, thick] (-5,2)--(-8,2);
\draw [loosely dotted, thick] (45,2)--(48,2);
\end{tikzpicture}}\]\caption{\label{figureARquiver}Auslander-Reiten quiver of $\Dd^b(\cov{H}{\partial})$ } \end{figure}
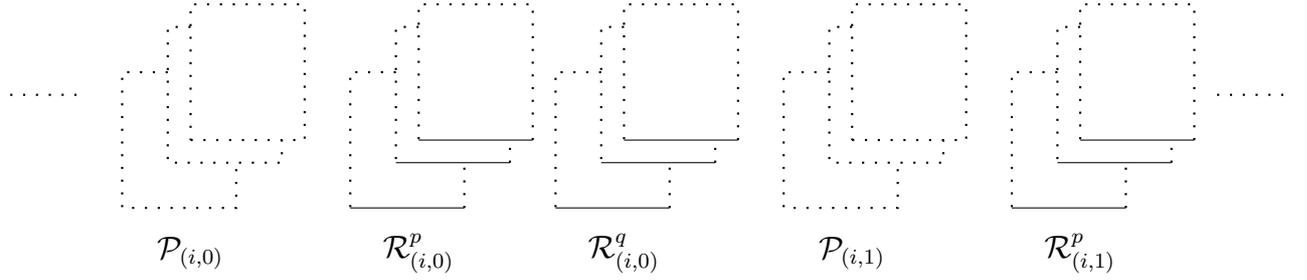

\bigskip
Now we can compute the Auslander-Reiten quiver of the orbit category $\Dd^b(\cov{H}{\partial})/\SSS\langle 1 \rangle$. We look at the action of the functor $\SSS\langle 1\rangle$ on the connected components of the AR-quiver of  $\Dd^b(\cov{H}{\partial})$. The functor $\SSS\langle 1 \rangle$ acts on the set $\{ \Pp_{(i,n)}, i\in\ZZ/w\ZZ, n\in\ZZ\}$ by  \[ \SSS(\mathcal{P}_{(i,n)})\langle 1\rangle = \tau ( \mathcal{P}_{(i,n)})[-1]\langle 1\rangle =\mathcal{P}_{(i,n)}[-1]\langle 1\rangle = \mathcal{P}_{(i+1,n-1)} \]
Then there are exactly $|w|$ orbits which are $(\SSS\langle 1 \rangle)^\ZZ \cdot \Pp_{(i,0)}$ for $i\in\ZZ/w\ZZ$.  
Similarly $\SSS\langle 1 \rangle$ acts on the sets $\{ \mathcal{R}^p_{(i,n)}, i\in\ZZ/w\ZZ, n\in\ZZ\}$ and $\{ \mathcal{R}^q_{(i,n)}, i\in\ZZ/w\ZZ, n\in\ZZ\}$ by  \[ \SSS(\mathcal{R}^p_{(i,n)})\langle 1\rangle = \mathcal{R}^p_{(i+1,n-1)}\ \textrm{and}\ \SSS(\mathcal{R}^q_{(i,n)})\langle 1\rangle = \mathcal{R}^q_{(i+1,n-1)}\]
Therefore we get $2|w|$ orbits which are $(\SSS\langle 1 \rangle)^\ZZ\cdot\mathcal{R}^p_{(i,0)}$ and  $(\SSS\langle 1 \rangle)^\ZZ\cdot\mathcal{R}^q_{(i,0)}$ for $i\in\ZZ/w\ZZ$.

Finally the shape of the AR-quiver of $\Dd^b(\cov{H}{\partial})/\SSS\langle 1 \rangle$ is described  in Figure~\ref{figureARquiver2}
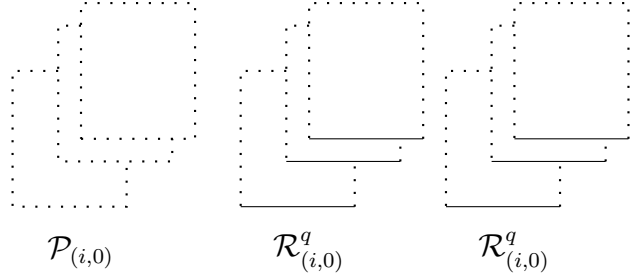
\begin{figure}
\[ \begin{tikzpicture}[>=stealth,scale=.3]
\draw [loosely dotted, thick] (0,0) --(5,0)--(5,6)--(0,6)--(0,0);
\draw [loosely dotted, thick] (0,5)--(-1,5)--(-1,-1)--(4,-1)--(4,0);
\draw [loosely dotted, thick] (-1,3)--(-3,3)--(-3,-3)--(2,-3)--(2,-1);
\node (P0) at (0,-5) {$\mathcal{P}_{(i,0)}$};
\draw [loosely dotted, thick] (15,0)--(15,6)--(10,6)--(10,0);
\draw (10,0)--(15,0);
\draw [loosely dotted, thick] (10,5)--(9,5)--(9,-1);
\draw [loosely dotted, thick] (14,-1) -- (14,0);
\draw (9,-1)--(14,-1);
\draw [loosely dotted, thick] (9,3)--(7,3)--(7,-3);
\draw [loosely dotted, thick] (12,-3)--(12,-1);
\draw (7,-3)--(12,-3);
\node (P2) at (10,-5) {$\mathcal{R}_{(i,0)}^q$};
\draw [loosely dotted, thick] (24,0)--(24,6)--(19,6)--(19,0);
\draw (19,0)--(24,0);
\draw [loosely dotted, thick] (19,5)--(18,5)--(18,-1);
\draw [loosely dotted, thick] (23,-1)--(23,0);
\draw (18,-1)--(23,-1);
\draw [loosely dotted, thick] (18,3)--(16,3)--(16,-3);
\draw [loosely dotted, thick] (21,-3)--(21,-1);
\draw (16,-3)--(21,-3);
\node (P1) at (19,-5) {$\mathcal{R}_{(i,0)}^q$};
\end{tikzpicture} \]
\caption{Auslander-Reiten quiver of the category $\Dd^b\Lambda$, where $\Lambda$ is of cluster type $\A_{p,q}$ and of weight $w>0$.} \label{figureARquiver2}
\end{figure}
This description can be summarized in the following.
\begin{cora}\label{shapeARquiver}
Let $\Lambda$ be an algebra of cluster type $\A_{p,q}$ and of weight  $w\neq 0$. Then the Auslander-Reiten quiver of $\Dd^b(\Lambda)$ has exactly $3|w|$ connected components:
\begin{itemize}
\item $|w|$ components of type $\ZZ A_\infty^\infty$;
\item $2|w|$ components of type $\ZZ A_\infty$.
\end{itemize}
\end{cora}

\subsection{Consequences of the description of the AR-quiver}

 From this description we also get the following consequences.
\begin{cora}
Let $\Lambda$ be an algebra of cluster type $\A_{p,q}$ and of weight $w\neq 0$. Then $\Lambda$ is representation-finite.
\end{cora}
\begin{proof}
The algebra $\Lambda$ is derived discrete in the sense of \cite{Vos01}. Hence it is representation-finite.
\end{proof}

\begin{rema}
Let $\Lambda$ be an algebra of cluster type $\A_{p,q}$ and of weight $w\neq 0$. Then the category $\Dd^b(\Lambda)$ is \emph{locally fractionally Calabi-Yau} of dimensions $\frac{p-2w}{p-w}$ and $\frac{q+2w}{q+w}$ in the sense that there exists objects $X$ and $Y$ such that there are isomorphisms $\mathbb{S}^{p-w}X\iso X[p-2w]$ and $\mathbb{S}^{q+w}Y\iso Y[q+2w]$. Such algebras are studied in \cite[Section~6]{AO}. 
\end{rema}

From the previous section, we also deduce a result for the image of the derived category in the cluster category.
\begin{cora}
Let $\Lambda$ be an algebra of cluster type $\A_{p,q}$ and of weight $w\neq 0$. Then the quiver of the orbit category $\Dd^b(\Lambda)/\SSS$ has 3 connected components which are of the forms $\ZZ\A_{p,q}$, $\ZZ A_\infty/(\tau^p)$ and $\ZZ A_\infty/(\tau^q)$.
\end{cora}

\begin{proof}
From Corollary~\ref{k-equivalence} we have the following diagram
\[\xymatrix{\Dd^b(\cov{\Lambda}{\delta}) \ar[rr]^\sim \ar[d] &&  \Dd^b(\cov{H}{\partial})\ar[d] \\ \Dd^b(\Lambda)\ar[d]\ar[rr]^-\sim && \Dd^b(\cov{H}{\partial})/\SSS\langle 1\rangle\ar[d] \\ \Dd^b(\Lambda)/\SSS\ar@{.>}[rr] && (\Dd^b(\cov{H}{\partial})/\SSS \langle 1 \rangle)/\SSS }\]
The upper functor is a triangle functor, hence it commutes with $\SSS$. Therefore, the $k$-equivalence of Corollary~\ref{k-equivalence} commutes with $\SSS$. Using this we deduce that we have a  $k$-equivalence 
\[\xymatrix{ \Dd^b(\Lambda)/\SSS\ar[rr]^-\sim && (\Dd^b(\cov{H}{\partial})/\SSS \langle 1 \rangle)/\SSS }\]
Therefore the AR quiver of $ \Dd^b(\Lambda)/\SSS$ is the same as the AR quiver of the orbit category $(\Dd^b(\cov{H}{\partial})/\SSS \langle 1 \rangle)/\SSS $. Note that since $\SSS\iso \langle -1 \rangle$ in $(\Dd^b(\cov{H}{\partial})/\SSS \langle 1 \rangle)$, we just have to understand the action of $\langle 1\rangle$ in the category $(\Dd^b(\cov{H}{\partial})/\SSS \langle 1 \rangle)$.
We know that the category $(\Dd^b(\cov{H}{\partial})/\SSS \langle 1 \rangle)$ has $3|w|$ components which are \[\Pp_{(i,0)}, \mathcal{R}^p_{(i,0)}\ \textrm{and } \mathcal{R}_{(i,0)}^q, \textrm{for } i\in\ZZ/w\ZZ\]
Moreover we have the equality \[ \Pp_{(i,0)}\langle 1 \rangle = \Pp_{(i+1,0)}, \mathcal{R}^p_{(i,0)}\langle 1 \rangle = \mathcal{R}^p_{(i+1,0)},\ \textrm{and } \mathcal{R}^q_{(i,0)}\langle 1 \rangle = \mathcal{R}^q_{(i+1,0)}\] Thus for each connected component $\Gamma$ of the AR quiver of $(\Dd^b(\cov{H}{\partial})/\SSS \langle 1 \rangle)$ we have $ \Gamma\langle w \rangle = \Gamma$. Now it is easy to check that for $X\in \mathcal{R}_{(i,0)}$ we have $X\langle w\rangle\iso \tau ^{-p}X$. 
The action of $\langle w\rangle$ on the component $\Pp_{(0,0)}$ makes $\Pp_{(0,0)}/\langle w \rangle$ isomorphic to $\ZZ \A_{p,q}$. Therefore we get the result.
\end{proof}

\section{Explicit description of algebras of cluster type  $\A_{p,q}$}\label{section explicit}

In this section, we explicitly describe (in terms of quivers with relations) the algebras of cluster type $\A_{p,q}$, and give an easy method for computing the weight of such an algebra. The strategy consists in describing first the cluster-tilted algebras of type $\A_{p,q}$, following \cite{Bas09}, and then showing that all algebras of cluster type $\A_{p,q}$ can been seen as the degree zero part of a cluster-tilted algebra of type $\A_{p,q}$ with an appropriate grading. We first start by describing the algebras of cluster type $A_n$ using a similar approach.

\subsection{Algebras of cluster type $A_n$}
The aim of this section is to describe all algebras of global dimension~$\leq 2$  which are of cluster type $A_n$. By Corollary~\ref{cor_tree}, we already know that they are the algebras of global dimension~$\leq 2$ derived equivalent to $kQ$, where $Q$ is a quiver of type $A_n$. A description of iterated-tilted algebras of type $A_n$ was done by Assem \cite{Ass82}. Here we use other techniques, based on further developments in cluster-tilting theory. 

We start with some definitions.
\begin{dfa}
Let $Q$ be a quiver. A cycle $a_1\ldots a_r$ in $Q$ is called \emph{irreducible} if for all $1\leq i\neq j\leq r$ we have $s(a_i)\neq s(a_j)$. All cycles which are not irreducible are called \emph{reducible}.
\end{dfa}

\begin{dfa}
For $n\geq 1$, we define the class $\Mm^A_n$ of quivers $Q$ that satisfy the following:
\begin{itemize}
\item they have $n$ vertices;
\item all non-trivial irreducible cycles are oriented and of length 3;
\item a vertex has valency at most 4;
\item if a vertex has valency 4, then two of its adjacent arrows belong to one 3-cycle, and the other two belong to another 3-cycle;
\item if a vertex has valency 3, then two of its adjacent arrows belong to a 3-cycle, and the third arrow does not belong to any 3-cycle.
\end{itemize}
The set $\Mm^A$ denotes the union of all $\Mm^A_n$.
For a quiver $Q$ in $\Mm_A$ we denote by $W_Q$ the sum of all oriented 3-cycles. This is a potential in the sense of \cite{DWZ}.
\end{dfa}
\begin{thma}[\cite{BV08}] \label{dagfinn}
Let $\Gamma$ be a basic finite dimensional algebra. Then $\Gamma$ is a cluster-tilted algebra of type $A_n$ if and only if there exists a quiver $Q$ in $\Mm^A_n$ such that we have an isomorphism $\Gamma\iso \Jac(Q,W_Q).$
\end{thma}

The main result of this subsection is the following.
\begin{thma}\label{algebraclustertypeA}
Let $Q$ be a quiver in $\Mm^A_n$, and let $d$ be a $W_Q$-grading (cf Definition~\ref{defWgrading}). It yields a grading on the Jacobian algebra $B:=\Jac(Q,W)$. Denote by $\Lambda:=B_0$ its degree zero part. Then $\Lambda$ is an algebra of global dimension $\leq 2$ and of cluster type $A_n$. Moreover each basic algebra of global dimension $\leq 2$ and of cluster type $A_n$ is isomorphic to such a $\Lambda$.
\end{thma}

\begin{proof}
We divide the proof into several steps for the convenience of the reader.

\medskip
\textit{Step 1: The global dimension of the algebra $\Lambda$ is at most 2.}

\smallskip
Denote by $Q^{(0)}$ the subquiver of $Q$ which is defined by $Q^{(0)}_0:=Q_0$ and $Q^{(0)}_1:=\{a\in Q_1, d(a)=0\}$. Then we have $$\Lambda\iso kQ^{(0)}/\langle \partial_aW_Q, d(a)=1\rangle.$$
Note that since $W$ is a sum of 3-cycles, the ideal of relations $\langle \partial_aW_Q, d(a)=1\rangle$ is contained in $kQ^{(0)}_2$, that is every relation is of length 2. 
Let $i\in Q_0$ be a vertex and $e_i$ be the associated primitive idempotent of $\Lambda$. Then the projective resolution of the simple $\Lambda$-module associated to $i$ is given by 
$$\xymatrix{\bigoplus_{b,s(b)=i,d(b)=1} e_{t(b)}\Lambda\ar[rr]^{(f_{a,b})} && \bigoplus_{a,t(a)=i, d(a)=0}e_{s(a)}\Lambda \ar[r] & e_i\Lambda\ar[r] & S_i\ar[r] & 0}$$
Let $b\colon i\rightarrow l$ be an arrow of degree 1, and $a\colon j\rightarrow i$ such that $f_{a,b}\neq 0$. Then there exists a 3-cycle:
$$
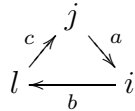
\xymatrix@-.5cm{ & j\ar[dr]^a & \\ l\ar[ur]^c & & i\ar[ll]^b}$$ and the map $f_{a,b}\colon e_l\Lambda\rightarrow e_j\Lambda$ is induced by $c$. The arrow $c$ does not belong to another 3-cycle, therefore there is no relation between some predecessor of $l$ and  $j$. Hence the map $f_{a,b}\colon e_l\Lambda\rightarrow e_j\Lambda$  is a monomorphism. 

Let $a$ be an arrow with $t(a)=i$ and of degree $0$. Then $a$ belongs to at most one cycle. Therefore there exists at most one arrow $b$ of degree 1 with $s(b)=i$ such that $f_{a,b}$ does not vanish. Therefore the map $$\xymatrix{\bigoplus_{b,s(b)=i,d(b)=1} e_{t(b)}\Lambda\ar[rr]^{(f_{a,b})} && \bigoplus_{a,t(a)=i, d(a)=0}e_{s(a)}\Lambda}$$ is injective, and the projective dimension of $S_i$ is at most 2.

\medskip
\textit{Step 2: We have an isomorphism $\Jac(Q,W_Q)\iso \End_{\Cc_\Lambda}(\pi\Lambda)$.}

\smallskip 
Since $d$ is a $W_Q$-grading the set $\{\partial_aW_Q, d(a)=1\}$ is a set of minimal relations for $\Lambda$. Therefore by Theorem  \ref{keller}, we have $\overline{Q}^{(0)}=Q$ and $W_\Lambda=W_Q$ and hence Step 2.

\medskip
\textit{Step 3: There is a triangle equivalence $\Cc_\Lambda\iso \Cc_{A_n}$}

\smallskip
By Theorems~\ref{dagfinn}, \ref{connectivity}, and \ref{birs} there exists a sequence of mutation $s$ such that $\mu_s(Q,W_Q)=(Q_H,0)$. Therefore we are done by Corollary~\ref{recognitioncor}.

\medskip
\textit{Step 4: Each algebra of cluster type $A_n$ and of global dimension $\leq 2$ is isomorphic to the degree zero part of a graded Jacobian algebra $\Jac(Q,W_Q,d)$ where $Q\in\Mm_n^A$ and $d$ is a $W_Q$-grading.}

\smallskip
Let $\Lambda\iso kQ_\Lambda/I$ be an algebra of global dimension $\leq 2$ and of cluster type $A_n$. Denote by $f\colon \Cc_\Lambda\rightarrow \Cc_Q$ the triangle equivalence. By Proposition~\ref{propWgrading} there exists a graded QP $(Q_{\overline{\Lambda}},W,d)$, where $W$ is rigid and $d$ is a $W$-grading, such that we have \[\overline{\Lambda}\underset{\ZZ}{\iso} \Jac(Q_{\overline{\Lambda}},W,d).\] Now  the object $f(\pi_\Lambda(\Lambda))$ is a cluster-tilting object in $\Cc_{A_n}$. Thus by Theorem~\ref{dagfinn}, there exists $Q\in\Mm^A_n$ such that we have an isomorphism \[\overline{\Lambda}\iso \Jac(Q,W_Q).\]
It is clear that we have $Q=Q_{\overline{\Lambda}}$. We then conclude using the following lemma.
\end{proof}

\begin{lema}\label{uniqueness_rigid_potential_A}
Let $(Q,W,d)$ be a graded QP where $Q\in\Mm^A_n$ and $W$ is rigid. Then $d$ makes $W_Q$ homogeneous of degree $1$ and we have an isomorphism of $\ZZ$-graded algebras
$$\Jac(Q,W_Q,d)\underset{\ZZ}{\iso}\Jac(Q,W,d).$$ 
\end{lema}

\begin{proof}
Denote by $C_1,\ldots,C_l$ the oriented 3-cycles of $Q$, so that we have $W_Q=\sum_{i=1}^r C_i$. One easily checks that the irreducible oriented cycles of the quiver $Q$ are exactly the $C_i$'s (up to cyclic equivalence). Therefore we can assume that \[W=\sum_{i=1}^r\lambda_i C_i + \textrm{extra terms}\] where the extra terms are reducible. If there exists $i$ with $\lambda_i= 0$ then the cycle $C_i$ and all the cycles cyclically equivalent to it are not in the Jacobian ideal since there are no cycle of length $\leq 2$ in the quiver $Q$. Therefore since $W$ is rigid, we have $\lambda_i\neq 0$ for all $i=1,\ldots,r$. Now the existence of the grading $d$ implies that $d(C_i)=1$ for all $i=1,\ldots,r$, so any reducible cycle is of degree at least 2. Hence the extra terms in the potential have to be zero and we have $W=\sum_{i=1}^r\lambda_i C_i$ with $\lambda_i\neq 0$.
\end{proof}

\begin{rema}
Note that Theorem~\ref{algebraclustertypeA} is not true for the other Dynkin types. Let $(Q,d)$ be the following graded quiver 
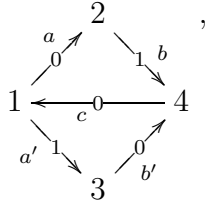
\[\xymatrix@-.2cm{& 2\ar[dr]|1^(.6)b& \\ 1\ar[ur]|0^(.6)a\ar[dr]|1_(.4){a'}&& 4\ar[ll]|0^(.6)c\\ &3\ar[ur]|0_(.4){b'}},\] and $W:=cba+cb'a'$. The Jacobian algebra $\Jac(Q,W)$ is a cluster-tilted algebra of type $D_4$. The grading $d$ is a $W$-grading but the degree zero part of the graded algebra $\Jac(Q,W,d)$ is an iterated-tilted algebra of type $A_4$ and of global dimension 3, so it cannot be of cluster type $D_4$. 
\end{rema}

\subsection{Cluster-tilted algebras of type $\widetilde{A}_{p,q}$}
\begin{dfa}[\cite{Bas09}]
For $p\geq q\geq 1$, we define the class $\Mm^{\widetilde{A}}_{p,q}$ of quivers $Q$ that satisfy the following conditions:
\begin{itemize}
\item $Q$ has $p+q$ vertices;
\item there exist integers $1\leq p_1<p_2< \ldots < p_r\leq p$ and $1\leq q_1<q_2<\ldots<q_r\leq q$ such that $Q$ contains precisely one full subquiver $C$ which is a non-oriented cycle of type $(p_1,q_1,p_2-p_1,q_2-q_1,\dots, p_r-p_{r-1},q_r-q_{r-1})$ (that is $C$ is the composition of $p_1$ arrows going in one direction with $q_1$ arrows going in the other direction, with $p_2-p_1$ arrows going the first direction, etc...). We denote by $a_1,\ldots, a_{p_r}$ the arrows of $C$ going in one direction and we call them the \emph{$p$-arrows}. We denote by $b_1,\ldots, b_{q_r}$ the arrows going in the opposite direction and we call them the \emph{$q$-arrows}.

\item each arrow connecting $C$ to a vertex not in $C$ is in exactly one 3-cycle of $Q$ of the form \[\scalebox{.7}
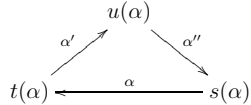
{\xymatrix{&u(\alpha)\ar[dr]^{\alpha''}&\\ t(\alpha)\ar[ur]^{\alpha'} && s(\alpha)\ar[ll]_{\alpha}}}\] where $\alpha$ is in $C$. We denote by $u(\alpha$) the connecting vertex. It has valency at most 4. When its valency is 4, the adjacent arrows which are not $\alpha'$ and $\alpha''$ belong to exactly one 3-cycle, and when it has valency 3, the third arrow does not belong to any 3-cycle. Moreover, the subquiver containing $C$, $\alpha'$ and $\alpha''$ is a full subquiver of $Q$. Hence we cannot have $u(\alpha)=u(\beta)$ for $\alpha\neq \beta$.

\item the full subquiver of $Q$ whose vertices are not in $C$ is a disjoint union of quivers $Q^\alpha\in\Mm^A$ where $\alpha$ is an arrow of the non-oriented cycle. The quiver $Q^\alpha$ is empty if there is no $3$-cycle containing $\alpha$, and the quiver $Q^\alpha$ contains the vertex $u(\alpha)$ if there is a 3-cycle 
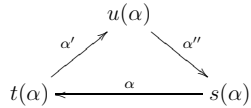
\[\scalebox{.7}{\xymatrix{&u(\alpha)\ar[dr]^{\alpha''}&\\ t(\alpha)\ar[ur]^{\alpha'} && s(\alpha)\ar[ll]_{\alpha}}}\] 
\begin{figure}

\[\scalebox{.7}{
\begin{tikzpicture}[>=stealth,scale=1]
\node (P1) at (-6,0){$.$};
\node (P2) at (-5.6,2){$.$};
\node (P3) at (-4.8,4){$.$};
\node (P4) at (-2,5.6){$.$};
\node (P5) at (0,6){$.$};
\node (P6) at (2,5.6){$.$};
\node (P7) at (4.8,4){$.$};
\node (P8) at (5.6,2){$.$};
\node (P9) at (6,0){$.$};
\node (P10) at (0,-6){$.$};
\node (P11) at (-2,-5.6){$.$};
\node (P12) at (-4,-4.8){$.$};
\node (P13) at (-5.6,-2){$.$};
\node (P14) at (-7.2,3.6)  {$u(a_2)$};
\node (P15) at (1.4,8)  {$u(b_2)$};
\node (P16) at (7,3.8)  {$u(b_{q_1})$};
\node (P17) at (-10,3.4)  {$Q^{a_2}$};
\node (P18) at (1.4,9.2) {$Q^{b_1}$};
\node (P19) at (9.6,3.8) {$Q^{b_{q_1}}$};

\draw [->] (P1) -- node [swap,xshift=3mm] {$a_1$}(P2);
\draw [->] (P2) -- node [swap,xshift=3mm, yshift=-1mm] {$a_2$}(P3);
\draw [->] (P4) -- node [swap,xshift=-1mm,yshift=-3mm] {$a_{p_1}$}(P5);
\draw [->] (P6) -- node [swap,yshift=-2mm] {$b_1$}(P5);
\draw [->] (P8) -- node [swap,xshift=-2mm] {$b_{q_1}$}(P7);
\draw [->] (P8) -- node [swap,xshift=5mm] {$a_{1_1+1}$}(P9);
\draw [->] (P10) -- node [swap,xshift=2mm,yshift=2mm] {$a_{p_r}$}(P11);
\draw [->] (P12) -- node [swap,xshift=-2mm,yshift=-2mm] {$b_{q_{r-1}+1}$}(P11);
\draw [->] (P1) -- node [swap,xshift=3mm] {$b_{q_r}$}(P13);

\draw [->] (P3) -- (P14);
\draw [->] (P14) -- (P2);
\draw [->] (P5) -- (P15);
\draw [->] (P15) -- (P6);
\draw [->] (P7) -- (P16);
\draw [->] (P16) -- (P8);
\draw (-9,3.4) [dotted] ellipse (3cm and 2cm);
\draw (1.4,9.2) [dotted] ellipse (3cm and 2cm);
\draw (9,3.8) [dotted] ellipse (3cm and 2cm);
 \draw [loosely dotted, thick] (P3) .. controls (-4,4.8) and (-3,5.4) .. (P4);
\draw [loosely dotted, thick] (P6) .. controls (3.4,5.2)  .. (P7);
\draw [loosely dotted, thick] (P9) .. controls (6,-2) and (3,-6)  .. (P10);
\draw [loosely dotted, thick] (P12) .. controls (-5.4,-3)  .. (P13);
\end{tikzpicture}}
\]\caption{\label{figuretildeA}Shape of a quiver in $\Mm^{\A}_{p,q}$. }
\end{figure}
\item We have the equalities $$p= \sum_{l=1}^{p_r} \sharp Q_0^{a_l} + p_r \quad \textrm{and}\quad q=\sum_{l=1}^{q_r} \sharp Q_0^{b_l} + q_r.$$
\end{itemize}
For a quiver $Q$ in $\Mm^{\widetilde{A}}_{p,q}$ we denote by $W_Q$ the sum of all oriented 3-cycles. More precisely we define \[ W_Q:=\sum_{\alpha \in C}(\alpha''\alpha'\alpha+W_{Q^\alpha})\] where $C$ is the non-oriented cycle of $Q$. This is a rigid potential in the sense of \cite{DWZ}. 
\end{dfa}

\begin{thma}[\cite{Bas09}] \label{bastian}
Let $\Gamma$ be a finite dimensional algebra. Then $\Gamma$ is a cluster-tilted algebra of type $\widetilde{A}_{p,q}$ if and only if there exists a quiver $Q$ in $\Mm^{\widetilde{A}}_{p,q}$ such that we have an isomorphism $$\Gamma\iso \Jac(Q,W_Q).$$
\end{thma}

\subsection{Algebras of cluster type $\A_{p,q}$}
We have the same kind of result as for the $A_n$ case.
\begin{thma}\label{algebraclustertypetildeA}
Let $Q$ be a quiver in $\Mm^{\A}_{p,q}$, and let $d$ be $W_Q$-grading. It yields a grading on the Jacobian algebra $B:=\Jac(Q,W_Q)$. Denote by $\Lambda:=B_0$ its degree zero part. Then $\Lambda$ is an algebra of global dimension $\leq 2$ and of cluster type $\A_{p,q}$. Moreover each basic algebra of global dimension $\leq 2$ and of cluster type $\A_{p,q}$ is isomorphic to such a $\Lambda$.
\end{thma}

 The proof of the first assertion is exactly the same as in the proof of Theorem \ref{algebraclustertypeA} (Steps~1, 2, and 3). For the proof of  the second assertion, we will need the following.

\begin{lema} \label{lemma_compatibility}
Let $(Q,W,d)$ be a graded quiver with reduced potential such that:
\begin{itemize}
\item the quiver $Q$ is in $\Mm^{\A}_{p,q}$,
\item the potential $W$ is rigid,
\item the grading $d$ is a $W$-grading.
\end{itemize}
Then there exists an algebra isomorphism $\varphi\colon k\hat{Q}\rightarrow k\hat{Q}$ (where $k\hat{Q}$ is the completion of the path algebra $kQ$)  such that $\varphi$ is the identity on the vertices, and such that $\varphi(W_Q)$ is cyclically equivalent to $W$. Moreover there exists a $W_Q$-grading $d'$ on $Q$ such that $\varphi\colon (k\hat{Q},d')\rightarrow (k\hat{Q},d)$ is an isomorphism of graded algebras.
\end{lema}

\begin{rema}
This means that the graded QP $(Q,W,d)$ and $(Q,W_Q,d')$ are graded right equivalent in the sense of~\cite[Def. 6.3]{AO10}.
\end{rema}

\begin{proof}
We denote by $C_1,\ldots, C_l$ the oriented cycles such that $W_Q=\sum_{i=1}^lC_i$, by $a_1,\ldots, a_{p_r}$ the $p$-arrows, and by $b_1,\ldots,b_{q_r}$ the $q$-arrows of $Q$. The $C_i$ are irreducible cycles of $Q$, but contrary to the $A_n$-case, there might be other irreducible cycles in the quiver $Q$. We treat here the most complicated case.

\medskip \noindent
\emph{Assume that for all $i=1, \ldots, p_r$, and all $j=1,\ldots, q_r$ we have $Q^{a_i}\neq \varnothing$ and $Q^{b_j}\neq \varnothing$.}

\smallskip
 Denote by $C_a$ a cycle containing exactly once the arrows $a_i, i=1,\ldots, p_r$ and $b'_j,b''_j, j=1,\ldots,q_r$ and by $C_b$ a cycle containing exactly once the arrows $b_j, j=1,\ldots, q_r$ and $a'_i,a''_i, i=1,\ldots,p_r$. Then one can check that the irreducible cycles of $Q$ are $C_a$, $C_b$ and the $C_i$'s up to cyclic equivalence. Therefore we can write \[
W=\sum_{i=1}^{l}\lambda_i C_i +\alpha C_a+ \beta C_b +\textrm{ extra terms},\]
where $\lambda_i,\alpha,\beta\in k$ and the extra terms are linear combinations of reducible cycles.

\medskip\noindent
First we show that we can assume $\lambda_i\neq 0$ for all $i=1,\ldots, l$. 

If $p_r+q_r\geq 3$, then we have immediately $\lambda_i\neq 0$ for $i=1,\ldots, l$ by the rigidity of $W$.
Assume $p_r=q_r=1$. Then $Q$ is of this form:

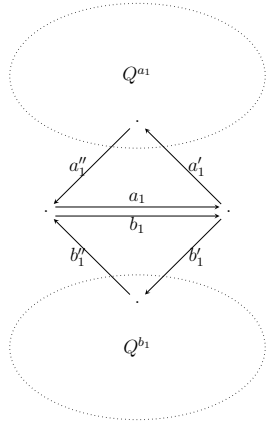
\[ \scalebox{.6}{
\begin{tikzpicture}[scale=1,>=stealth]
\node (P1) at (0,0)  {$.$};
\node (P2) at (4,0)  {$.$};
\node (P3) at (2,2)  {$.$};
\draw [->] (0.2,0.1) -- node [swap,yshift=2mm] {$a_1$} (3.8,0.1);
\draw [->] (P2) -- node [swap,xshift=3mm] {$a'_1$} (P3);
\draw [->] (P3) -- node [swap, xshift=-3mm] {$a''_1$} (P1);
\node (P4) at (2,3) {$Q^{a_1}$};
\draw [dotted] (2,3) ellipse (2.8cm and 1.6cm);

\node (P5) at (2,-2){$.$};
\node (P6) at (2,-3) {$Q^{b_1}$};
\draw [->] (0.2,-0.1) -- node [swap,yshift=-2mm] {$b_1$} (3.8,-0.1);
\draw [->] (P2) -- node [swap,xshift=3mm] {$b'_1$} (P5);
\draw [->] (P5) -- node [swap, xshift=-3mm] {$b''_1$} (P1);
\draw [dotted] (2,-3) ellipse (2.8cm and 1.6cm);
\end{tikzpicture}}
\]

We can assume (up to renumbering) that $C_1$ is cyclically equivalent to $a_1a_1'a_1''$ and $C_2$ is cyclically equivalent to $b_1b_1'b_1''$, and we have $C_a=a_1b_1'b_1''$ and $C_b=b_1a'_1a''_1$. Then it is easy to see that the rigidity of $W$ implies that $\lambda_i\neq 0$ for $i=3,\ldots,l$ and that  $\lambda_1\lambda_2-\alpha\beta\neq 0$. Therefore, up to the automorphism of $Q$ exchanging $a_1$ and $b_1$ we can assume that $\lambda_1\lambda_2\neq 0$.

\medskip
Since the potential $W$ is homogeneous of degree 1 we have \begin{equation}\label{eqdeg1} 
d(a_i)+d(a'_i)+d(a''_i)=d(b_j)+d(b_j')+d(b_j'')= 1 \quad \textrm{ for all } i=1,\ldots, p_r,\ j=1,\ldots, q_r\end{equation}

By definition we have \begin{equation} \label{eqdeg2} d(C_a)=\sum_{i=1}^{p_r}d(a_i)+ \sum_{j=1}^{q_r} (d(b'_j)+d(b''_j))\quad\textrm{and}\quad d(C_b)=\sum_{j=1}^{q_r}d(b_j)+ \sum_{i=1}^{p_r} (d(a'_i)+d(a''_i))\end{equation}

Hence combining \eqref{eqdeg1} and \eqref{eqdeg2} we get \begin{equation}\label{eqdeg3} d(C_a)+d(C_b)=p_r+q_r.\end{equation} Using \eqref{eqdeg3} and the fact that $d(C_a)$ and $d(C_b)$ are non-negative (since $d$ is a map $Q_1\rightarrow\{0,1\}$), we divide the proof into 4 subcases.

\medskip

\noindent
\emph{Case 1: $d(C_a)\geq 2$ and $d(C_b)\geq 2$.}

\smallskip
In this case, since $W$ is homogeneous of degree $1$, we have $\alpha=\beta=0$, and there is no extra term in the potential $W$. For $i=1,\ldots, l$ we denote by $c_i$ the arrow such that $C_i=c_i''c'_ic_i$. Then we define $\varphi$ on $Q_1$ by 
\[ \varphi(x)=\left\{\begin{array}{cc} 
\lambda_i c_i & \ \textrm{if } x=c_i \\
x & \textrm{otherwise}
\end{array}\right. \] 
 It is then clear that $\varphi$ is an isomorphism of the graded algebra $(kQ,d)$ sending $W_Q$ onto $W$.

\medskip

\noindent
\emph{Case 2: $d(C_a)=1$ and $d(C_b)\geq 2$.}

\smallskip
In this case, since $W$ is homogeneous of degree 1, we have $\beta=0$ and there is no extra term in the potential $W$. Up to cyclic equivalence and renumbering we can assume that  $C_1=b''_1b'_1b_1$ and $C_a=b_1''b_1'C'_a$. For $i=2,\ldots, l$ we denote by $c_i$ the arrow such that $C_i=c_i''c'_ic_i$. Now we define $\varphi$ on $Q_1$ by 
\[ \varphi(x)=\left\{\begin{array}{cl} \lambda_1 b_1 +\alpha C'_a & \ \textrm{if } x=b_1\\
\lambda_i c_i & \ \textrm{if } x=c_i , \textrm{ and }i\geq 2\\
x & \textrm{otherwise}
\end{array}\right. \] 
Since $C_a=b_1''b_1'C'_a$ and $C_1=b''_1b'_1b_1$ are oriented cycles, the path $C'_a$ has the same source and the same target as $b_1$, thus $\varphi$ is an algebra morphism. Moreover, since $\lambda_i\neq 0$, $\varphi$ is an isomorphism of the completion $k\hat{Q}$. 
Now, we have 
\[ d(C'_a)=d(C_a)-d(b'_1)-d(b''_1)= 1  -d(b'_1)-d(b''_1)=d(b_1).\]
Therefore $\varphi$ is  an isomorphism of the graded algebra $(k\hat{Q},d)$ sending $W_Q$ onto $W$.
\medskip

\noindent
\emph{Case 3: $d(C_a)=1$ and $d(C_b)=1$.}

\smallskip
From \eqref{eqdeg3} we automatically have $p_r=q_r=1$. All the cycles $C_1=a''_1a'_1a_1$, $C_2=b''_1b'_1b_1$, $C_b=a''_1a'_1b_1$, $C_a=b_1''b'_1a_1$, $C_3,\ldots,C_l$ are homogeneous of degree 1. For $i=3,\ldots, l$ we denote by $c_i$ the arrow such that $C_i=c_i''c'_ic_i$. Since $W$ is rigid we have $\lambda_1\lambda_2-\alpha\beta\neq 0$. And since the grading $d$ makes $W$ homogeneous of degree 1, there is no extra term in the potential $W$ and we have $d(a_1)=d(b_1)$. Then we can define: 
\[ \varphi(x)=\left\{\begin{array}{cc} 
 \lambda_1 a_1+\beta b_1 & \ \textrm{if } x=a_1 \\
 \lambda_2 b_1+ \alpha a_1 &\ \textrm{if } x=b_1 \\
 \lambda_i c_i & \ \textrm{if } x=c_i , \textrm{ and }i\geq 3\\
x & \textrm{otherwise}
\end{array}\right. \]
Since $\lambda_1\lambda_2-\alpha\beta=0$, this algebra morphism is an isomorphism. Moreover, since $d(a_1)=d(b_1)$, $\varphi\colon (kQ,d)\rightarrow (kQ,d)$ is an isomorphism of graded algebras. By construction it sends $W_Q$ onto $W$.

\medskip

\noindent
\emph{Case 4: $d(C_a)=0$.}

\smallskip
From \eqref{eqdeg3}, we have $d(C_b)\geq 2$. Hence, since $W$ is homogeneous of degree 1, we have $\alpha=\beta=0$.  In this case, since the degree of $C_a$ is 0, the oriented cycles of the quiver $Q$ which are homogeneous of degree 1 are cyclically equivalent to something of the form $C_iC_a^n$ where $n\in\mathbb{N}$. Thus we can write up to cyclic equivalence:
\[ W=\sum_{i=1}^lC_iP_i(C_a),\]
where $P_i\in k[\![X]\!]$ is a power series with constant term $\lambda_i\neq 0$. For each $i=1,\ldots, l$ we write $C_i=c''_ic'_ic_i$ and we define
\[ \varphi(x)=\left\{\begin{array}{cc} 
c_iP_i(C_a)& \ \textrm{if } x=c_i \\
x & \textrm{otherwise}
\end{array}\right. \]
Then $\varphi$ is an automorphism of the completion $k\hat{Q}$ since $\lambda_i\neq 0$ for all $i=1,\ldots, l$. Since $d(C_a)=0$, this automorphism is an automorphism of the graded algebra $(k\hat{Q},d)$.

\medskip
The other cases, namely when \emph{there exists $1\leq i \leq p_r$ and $1\leq j\leq q_r$ such that $Q^{a_i}= \varnothing$ and $Q^{b_j}=\varnothing$}, and \emph{for all $i=1,\ldots p_r$,  we have $Q^{a_i}\neq \varnothing$ and there exists $j$ such that $Q^{b_j}= \varnothing$} are simpler since in these cases there are less irreducible cycles. The proof is left to the reader.
\end{proof}

\begin{rema}
\begin{enumerate}
\item
Note that the first part of this lemma can be deduced directly from \cite{DWZ}. Indeed since $W$ is rigid the quiver with potential $(Q,W)$ is right equivalent in the sense of \cite{DWZ} to $(Q,W_Q)$ (there exists a sequence $s$ such that $\mu_s(Q)$ is acyclic, therefore $\mu_s(Q,W)$ is right equivalent to $\mu_s(Q,W_Q)$). By definition, this implies that there exists an automorphism of completed path algebras $\varphi\colon  k\hat{Q} \rightarrow k\hat{Q}$ which is the identity on the vertices, and such that $\varphi(W_Q)$ is cyclically equivalent to $W$.  However, we have proved this lemma constructing explicitly the isomorphism $\varphi$.
\item In the case $p_r=q_r=1$, it might happen that $\lambda_1\lambda_2=0$. Then the automorphism $\varphi$ constructed above will exchange the arrows $a_1$ and $b_1$. The degree map $d'$ will satisfy $d'(a_1)=d(b_1)$ and $d'(b_1)=d(a_1)$. This is the only case where $d$ and $d'$ are not the same.
\end{enumerate}
\end{rema}

From Lemma~\ref{lemma_compatibility} we deduce the following result, which is a restatement of the second part of Theorem~\ref{algebraclustertypetildeA} and finishes the proof of Theorem~\ref{algebraclustertypetildeA}.
\begin{cora}
Let $\Lambda$ be an algebra of global dimension at most $2$ and of cluster type $\A_{p,q}$. Then there exists a quiver $Q\in \Mm_{p,q}^{\A}$ and a $W_Q$-grading $d'$ such that $\Lambda$ is isomorphic to the degree zero part of $\Jac(Q,W_Q,d)$.
\end{cora}

\begin{proof}
By Proposition~\ref{propWgrading}, there exists a reduced graded quiver with potential $(Q,W,d)$, such that $\overline{\Lambda}\underset{\ZZ}{\sim} \Jac(Q,W,d)$. Moreover, $W$ is rigid and $d$ is a $W$-grading. By Theorem~\ref{bastian}, the quiver $Q$ is in $\Mm_{p,q}^{\A}$. By Lemma~\ref{lemma_compatibility}, there exists a $W_Q$-grading $d'$ such that we have $\Jac(Q,W,d)\underset{\ZZ}{\sim}\Jac(Q,W_Q,d')$. Therefore $\Lambda$ is isomorphic to the degree zero part of $\Jac(Q,W_Q,d')$.
\end{proof}

\begin{rema}
\begin{itemize}
\item[(1)] This corollary implies that an algebra $\Lambda$ of cluster type $\A_{p,q}$ is always isomorphic to an algebra of the form $kQ_\Lambda/I$, where the relations are paths of length~$2$. 

\item[(2)] This corollary gives a description of the iterated tilted algebras of global dimension $\leq 2$ of type $\A_{p,q}$. A description of all iterated tilted algebras (not distinguishing with respect to their global dimension) of type $\A_{p,q}$ has been given in \cite{AS87}.
\end{itemize} 
\end{rema}
 
\begin{cora}
There are only finitely many algebras (up to Morita equivalence) of global dimension $\leq 2$ and of cluster type $\A_{p,q}$.
\end{cora}
\begin{proof}
There are only finitely many quivers in the set $\Mm^{\A}_{p,q}$. And given $Q\in\Mm^{\A}_{p,q}$ there are finitely many $W_Q$-gradings.
\end{proof}

\subsection{An alternative description of the weight}
In this subsection, we give an explicit way to compute the weight of an algebra of cluster type $\A_{p,q}$.
\begin{dfa}
Let $Q$ be a quiver in $\Mm_{p,q}^{\A}$ and $d$ be a $\mathbb{Z}$-grading on $Q$ such that the potential $W_Q$ is homogeneous of degree 1. Then let $a_1,\ldots,a_{p_r}$ be the $p$-arrows and $b_1,\ldots,b_{q_r}$ be the $q$-arrows of $Q$. The \emph{weight} of the graded quiver $(Q,d)$ is defined to be 
$$w(Q,W_Q,d):=\sum_{l=1}^{p_r}d(a_l)-\sum_{l=1}^{q_r}d(b_l).$$
\end{dfa}

The aim of the subsection is to show the following.

\begin{prop}\label{prop alternative description weight}
Let $Q\in\Mm_{p,q}^{\A}$ and $d$ be a $W_Q$ grading. Let $\Lambda$ be the degree zero part of the graded algebra $\Jac(Q,W_Q,d)$. Then $\Lambda$ is an algebra of global dimension at most 2 and of cluster type $\A_{p,q}$ and we have 
$$w(\Lambda)=w(Q,W_Q,d).$$
\end{prop}

\begin{proof}
The first part of the statement follows from Theorem~\ref{algebraclustertypetildeA}.

 Let $s$ be a a sequence of mutation such that $\mu_s(Q,W_Q)=(H,0)$ where $Q_H$ is an acyclic quiver of type $\A_{p,q}$. Define a grading $\partial$ on $Q_H$ by $\mu_s^{\rL}(Q,W_Q,d)=(H,0,\partial)$. By definition we have $w(\Lambda)=w(H,0,\partial)$. Hence the proof of the proposition comes directly from the following technical lemma. 
\end{proof}

\begin{lema}\label{winvariantmutation}
Let $Q$ be a quiver in $\Mm_{p,q}^{\A}$ and $d$ be a $\mathbb{Z}$-grading on $Q$ such that $W_Q$ is homogeneous of degree 1. Let $i$ be a vertex of $Q$. Then we have
$$\xymatrix{w(\mu_i^{\rL}(Q,W_Q,d))=w(Q,W_Q,d)}.$$
\end{lema}

\begin{proof}
We define the grading $d'$ on the quiver $\mu_i(Q)$ by $\mu_i^{\rL}(Q,W,d)=(\mu_i(Q,W),d')$.
We distinguish the following different cases.
\medskip

\textit{Case 1: There exists $\alpha$ in the non-oriented cycle $C$ of $Q$ such that $i\in Q^\alpha$ and $i\neq u(\alpha)$.}

\smallskip
\noindent
Then the vertices adjacent to $i$ are not in the non-oriented cycle. Therefore the mutation of $Q$ at $i$ does not affect the non-oriented cycle, and so the weight clearly remains the same.

\medskip

\textit{Case 2: $i=u(\alpha)$ for some arrow $\alpha$ which is on the non-oriented cycle $C$.}

\smallskip
\noindent
Assume that $\alpha$ is a $p$-arrow.
There exists a 3-cycle $\alpha''\alpha'\alpha$  which is a summand of $W$. 
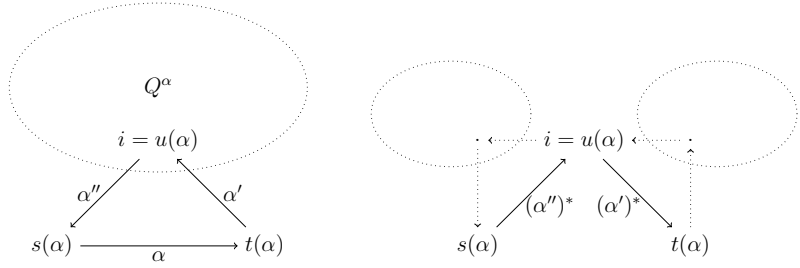
\[\scalebox{.7}{
\begin{tikzpicture}[scale=1]
\node (P1) at (0,0)  {$s(\alpha)$};
\node (P2) at (4,0)  {$t(\alpha)$};
\node (P3) at (2,2)  {$i=u(\alpha)$};
\draw [->] (P1) -- node [swap,yshift=-2mm] {$\alpha$} (P2);
\draw [->] (P2) -- node [swap,xshift=4mm] {$\alpha'$} (P3);
\draw [->] (P3) -- node [swap, xshift=-3mm] {$\alpha''$} (P1);
\node (P4) at (2,3) {$Q^\alpha$};
\draw [dotted] (2,3) ellipse (2.8cm and 1.6cm);
\node (P5) at (8,0)  {$s(\alpha)$};
\node (P6) at (12,0)  {$t(\alpha)$};
\node (P7) at (10,2)  {$i=u(\alpha)$};
\draw [->] (P7) -- node [xshift=-3.5mm,yshift=-2.5mm] {$(\alpha')^*$} (P6);
\draw [->] (P5) -- node [xshift=+3.5mm,yshift=-2.5mm] {$(\alpha'')^*$} (P7);
\node (P8) at (8,2){$.$};
\node (P9) at (12,2){$.$};
\draw [dotted] [->] (P8) -- (P5);
\draw [dotted] [->] (P7) -- (P8);
\draw [dotted] [->] (P6) -- (P9);
\draw [dotted] [->] (P9) -- (P7);
\draw [dotted] (7.5,2.5) ellipse (1.5cm and 1cm);
\draw [dotted] (12.5,2.5) ellipse (1.5cm and 1cm);

\end{tikzpicture}}
\]

The new arrows $(\alpha')^*$ and $(\alpha'')^*$ become $p$-arrows in the non-oriented cycle. Therefore we have 
$$\begin{array}{rcl} w(\mu_i^{\rL}(Q,W,d)) &= &w(Q,W,d)-d(\alpha) +d'((\alpha')^*)+d'((\alpha'')^*)\\ &= & w(Q,W,d)-d(\alpha) -d(\alpha')-d(\alpha'')+1\\ &= & w(Q,W,d).\end{array}$$  The last equality holds since $\alpha''\alpha'\alpha$ is a summand in the potential $W$, and hence we have $d(\alpha)+d(\alpha')+d(\alpha'')=1$.

\medskip
\textit{Case 3: $i$ is on the non-oriented cycle $C$ between two $p$-arrows.}
\smallskip

\noindent
Then the quiver $Q$ locally looks like
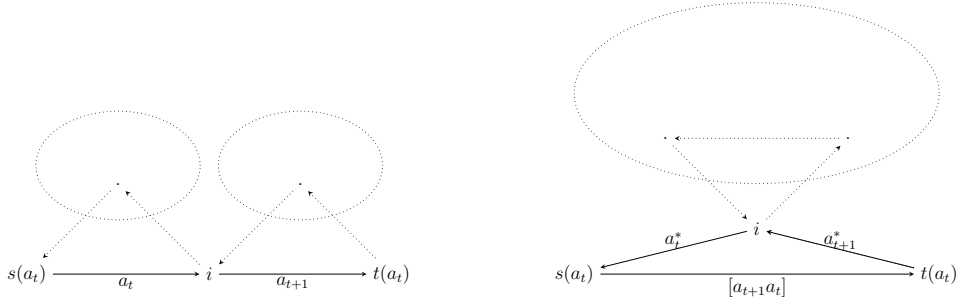
\[\scalebox{.6}{
\begin{tikzpicture}[scale=1,>=stealth]
\node (P1) at (0,0)  {$s(a_t)$};
\node (P2) at (4,0)  {$i$};
\node (P3) at (2,2){$.$};
\node (P5) at (8,0)  {$t(a_t)$};
\node (P4) at (6,2){$.$};
\draw [->] (P1) -- node [swap,yshift=-2mm] {$a_t$} (P2);
\draw [->] (P2) -- node [swap,yshift=-2mm] {$a_{t+1}$} (P5);
\draw [->] [dotted] (P3)--(P1);
\draw [->] [dotted] (P2)--(P3);
\draw [->] [dotted] (P5)--(P4);
\draw [->] [dotted] (P4)--(P2);
\draw [dotted] (2,2.4) ellipse (1.8cm and 1.2cm);
\draw [dotted] (6,2.4) ellipse (1.8cm and 1.2cm);

\node (Q1) at (12,0)  {$s(a_t)$};
\node (Q2) at (16,1)  {$i$};
\node (Q3) at (14,3){$.$};
\node (Q5) at (20,0)  {$t(a_t)$};
\node (Q4) at (18,3){$.$};
\draw [->] (Q2) -- node [swap,yshift=2mm] {$a_t^*$} (Q1);
\draw [->] (Q5) -- node [swap,yshift=2mm] {$a_{t+1}^*$} (Q2);
\draw [->] [dotted] (Q3) -- (Q2);
\draw [->] [dotted] (Q5) --(Q2);
\draw [->] [dotted] (Q2)--(Q4);
\draw [->][dotted] (Q4)--(Q3);
\draw [dotted] (16,4) ellipse (4cm and 2cm);
\draw [->] (Q1)--node [swap, yshift=-3mm] {$[a_{t+1}a_t]$}(Q5) ;
\end{tikzpicture}}
\]
 
The arrows $a_t$ and $a_{t+1}$ are replaced by the new $p$-arrow $[a_{t+1}a_t]$ in the non-oriented cycle. Hence we have
$$\begin{array}{rcl} w(\mu_i^{\rL}(Q,W,d))&=& w(Q,W,d)-d(a_t)-d(a_{t+1})+ d'([a_{t+1}a_t])\\ & = & w(Q,W,d)\end{array}.$$
The case where $i$ is between two $q$ arrows is similar.

\medskip
\textit{Case 4: $i$ is a sink of the non-oriented cycle $C$.}
\smallskip

\noindent
Assume that $i$ is the target of the $p$-arrow $a_l$ and the target of the $q$-arrow $b_t$.  
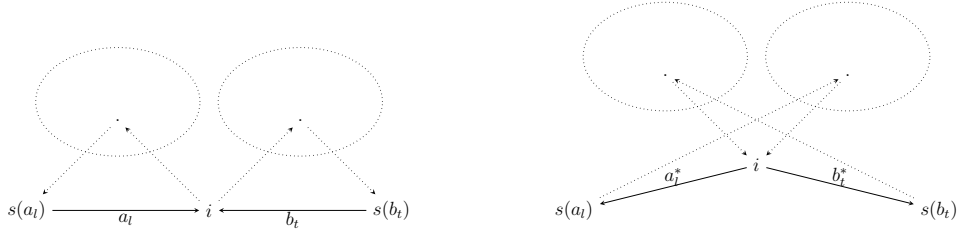
\[\scalebox{.6}{
\begin{tikzpicture}[scale=1,>=stealth]
\node (P1) at (0,0)  {$s(a_l)$};
\node (P2) at (4,0)  {$i$};
\node (P3) at (2,2){$.$};
\node (P5) at (8,0)  {$s(b_t)$};
\node (P4) at (6,2){$.$};
\draw [->] (P1) -- node [swap,yshift=-2mm] {$a_l$} (P2);
\draw [<-] (P2) -- node [swap,yshift=-2mm] {$b_{t}$} (P5);
\draw [->] [dotted] (P3)--(P1);
\draw [->] [dotted] (P2)--(P3);
\draw [->] [dotted] (P2)--(P4);
\draw [->] [dotted] (P4)--(P5);
\draw [dotted] (2,2.4) ellipse (1.8cm and 1.2cm);
\draw [dotted] (6,2.4) ellipse (1.8cm and 1.2cm);

\node (Q1) at (12,0)  {$s(a_l)$};
\node (Q2) at (16,1)  {$i$};
\node (Q3) at (14,3){$.$};
\node (Q5) at (20,0)  {$s(b_t)$};
\node (Q4) at (18,3){$.$};
\draw [->] (Q2) -- node [swap,yshift=2mm] {$a_l^*$} (Q1);
\draw [->] (Q2) -- node [swap,yshift=2mm] {$b_t^*$} (Q5);
\draw [->] [dotted] (Q3) -- (Q2);
\draw [->] [dotted] (Q4) --(Q2);
\draw [->] [dotted] (Q1)--(Q4);
\draw [->][dotted] (Q5)--(Q3);
\draw [dotted] (14,3.4) ellipse (1.8cm and 1.2cm);
\draw [dotted] (18,3.4) ellipse (1.8cm and 1.2cm);
\end{tikzpicture}}
\]
 
Then the arrows $a_l$ and $b_t$ are replaced by the arrows $a_l^*$ and $b_t^*$ and we have

$$\begin{array}{rcl} w(\mu_i^{\rL}(Q,W,d)) & = & w(Q,W,d)-d(a_l)+d(b_t)+d'(b_t^*)-d'(a_l^*)\\ & = & w(Q,W,d)-d(a_l)+d(b_t)+(-d(b_t)+1)-(-d(a_l)+1)\\ &=&w(Q,W,d).\end{array}$$
The case where $i$ is the source of one $p$-arrow and of one $q$-arrow is similar.
\end{proof}

This alternative description of the weight gives us the following consequences.
\begin{cora}
The number of derived equivalence classes of algebras of cluster type $\A_{p,q}$ is $[\frac{p}{2}]+[\frac{q}{2}] +1$ if $p\neq q$ and $[\frac{p}{2}]+1$ if $p=q$.
\end{cora}
\begin{proof}
Let $Q$ be a quiver in $\Mm_{p,q}^{\A}$. It is clear from the definition that $w$ is maximal when the $W_Q$-grading satisfies $d(a_l)=1$ for $l=1,\ldots,p_r$ and $d(b_j)=0$ for $j=1,\ldots, q_r$. In this case the weight is equal to $p_r=p-\sum_{l=1}^{p_r}\sharp Q^{a_l}_0$. But since $d$ is a $W_Q$-grading, $d(a_l)=1$ implies that the quiver $Q^{a_l}$ is non-empty. Then we have $w=p_r= p- \sum_{l=1}^{p_r}\sharp Q^{a_l}_0\leq p-p_r$. Thus we have $w\leq [\frac{p}{2}]$.
For the same reason, we have $w\geq -[\frac{q}{2}]$.  Now it is easy to see that all values $-[\frac{q}{2}],- [\frac{q}{2}]+1,\ldots ,[\frac{p}{2}]$ can occur.
\end{proof}

\begin{cora}\label{corollary cycle}
Let $\Lambda$ be an algebra of cluster type $\A_{p,q}$, which is not piecewise hereditary. Then there exists a tilting object $T$ in $\Dd^b(\Lambda)$ such that the quiver of $\End_{\Dd^b(\Lambda)}(T)$ has an oriented cycle.
\end{cora}

\begin{proof}
Let $-\frac{q}{2}\leq w\leq \frac{p}{2}$ be a non-zero integer. We construct an algebra $B$ of cluster type $\A_{p,q}$ of weight $w$ such that $Q_B$ has an oriented cycle.  Without loss of generality we can assume $w>0$. We define $B$ as follows.
 \[\scalebox{.8}
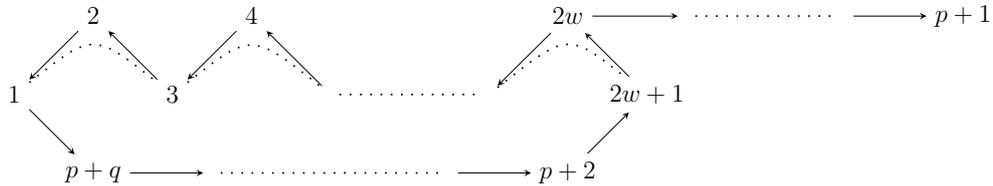
{
\begin{tikzpicture}[scale=1.3,>=stealth]
\node (P1) at (0,0)  {$1$};
\node (P2) at (1,1) {$2$};
\node(P3) at (2,0) {$3$};
\node(P4) at (3,1){$4$};
\node(P5) at (4,0){};
\node (P6) at (6,0){};
\node (P7) at (7,1) {$2w$};
\node (P8) at (8,0){$2w+1$};
\node (P9) at (8.5,1){};
\node (P10) at (10.5,1){};
\node (P11) at (12,1){$p+1$};
\node (P12) at (7,-1){$p+2$};
\node (P13) at (5.5,-1){};
\node(P14) at (1,-1){$p+q$};
\node(P15) at (2.5,-1){};

\draw [->] (P2)--(P1);
\draw[->] (P3)--(P2);
\draw[->] (P4)--(P3);
\draw[->] (P5)--(P4);
\draw[loosely dotted, thick] (P5)--(P6);
\draw[->] (P7)--(P6);
\draw[->] (P8)--(P7);
\draw[->] (P7)--(P9);
\draw[loosely dotted, thick] (P9)--(P10);
\draw[->] (P10)--(P11);
\draw[->] (P12)--(P8);
\draw[->] (P13)--(P12);
\draw[loosely dotted, thick] (P15)--(P13);
\draw[->] (P14)--(P15);

\draw[->](P1)--(P14);
 \draw [loosely dotted, thick] (P1) .. controls (1,.8) .. (P3);
 \draw [loosely dotted, thick] (P3) .. controls (3,.8) .. (P5);
 \draw [loosely dotted, thick] (P6) .. controls (7,.8) .. (P8);
\end{tikzpicture}}\] 
It is clear that $B$ is the degree 0 part of the Jacobian algebra $\Jac(Q,W_Q,d)$ with the graded quiver 
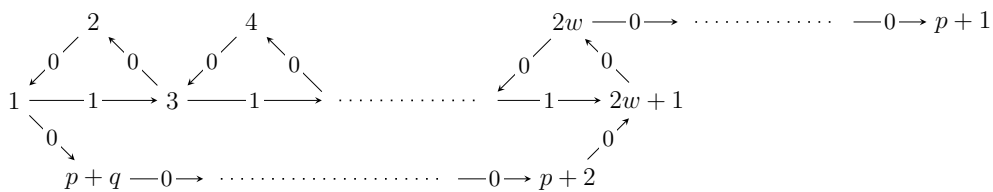
\[\scalebox{.8}{
\begin{tikzpicture}[scale=1.3,>=stealth]
\node (P1) at (0,0)  {$1$};
\node (P2) at (1,1) {$2$};
\node(P3) at (2,0) {$3$};
\node(P4) at (3,1){$4$};
\node(P5) at (4,0){};
\node (P6) at (6,0){};
\node (P7) at (7,1) {$2w$};
\node (P8) at (8,0){$2w+1$};
\node (P9) at (8.5,1){};
\node (P10) at (10.5,1){};
\node (P11) at (12,1){$p+1$};
\node (P12) at (7,-1){$p+2$};
\node (P13) at (5.5,-1){};
\node(P14) at (1,-1){$p+q$};
\node(P15) at (2.5,-1){};

\draw [->] (P2)-- node [fill=white,inner sep=.5mm]{\small{0}}(P1);
\draw[->] (P3)-- node [fill=white,inner sep=.5mm]{\small{0}}(P2);
\draw[->] (P4)-- node [fill=white,inner sep=.5mm]{\small{0}}(P3);
\draw[->] (P5)-- node [fill=white,inner sep=.5mm]{\small{0}}(P4);
\draw[loosely dotted, thick] (P5)--(P6);
\draw[->] (P7)-- node [fill=white,inner sep=.5mm]{\small{0}}(P6);
\draw[->] (P8)-- node [fill=white,inner sep=.5mm]{\small{0}}(P7);
\draw[->] (P7)-- node [fill=white,inner sep=.5mm]{\small{0}}(P9);
\draw[loosely dotted, thick] (P9)--(P10);
\draw[->] (P10)-- node [fill=white,inner sep=.5mm]{\small{0}}(P11);
\draw[->] (P12)-- node [fill=white,inner sep=.5mm]{\small{0}}(P8);
\draw[->] (P13)-- node [fill=white,inner sep=.5mm]{\small{0}}(P12);
\draw[loosely dotted, thick] (P15)--(P13);
\draw[->] (P14)-- node [fill=white,inner sep=.5mm]{\small{0}}(P15);

\draw[->](P1)-- node [fill=white,inner sep=.5mm]{\small{0}}(P14);
 \draw [->](P1) --  node [fill=white,inner sep=.5mm]{\small{1}}(P3);
 \draw [->](P3)--  node [fill=white,inner sep=.5mm]{\small{1}}(P5);
 \draw [->] (P6)-- node [fill=white,inner sep=.5mm]{\small{1}} (P8);
\end{tikzpicture}}\] 
Then by Theorem~\ref{derivedeqiffwegal} this algebra $B$ is the endomorphism algebra of some tilting complex $T\in\Dd^b(\Lambda)$.
\end{proof}

\begin{cora}
Let $\Lambda$ be an algebra of cluster type $\A_{p,q}$ and of weight $w$, then the Coxeter polynomial of $\Lambda$ is \[ X^{p+q}-(-1)^wX^{p-w}-(-1)^wX^{q+w}+1.\]
\end{cora}

\begin{proof}
By definition the Coxeter matrix is the matrix of the automorphism $\tau$ at the level of the Grothendieck group $K_0(\Dd^b(\Lambda))$ in the basis $\{[S_i] \mid 1 \leq i \leq p+q\}$ consisting of the representatives of the simples. The Coxeter polynomial $C(X)$ is its characteristic polynomial.

The result is already known for $w=0$. Without loss of generality we can assume that  $0< w\leq \frac{p}{2}$.  
 By the above results we can assume that $\Lambda$ is given by the following quiver with relations:
 \[\scalebox{.8}
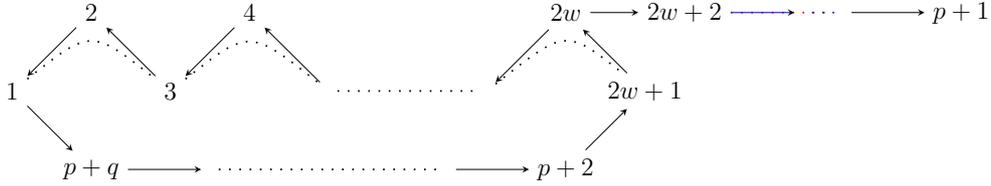
{
\begin{tikzpicture}[scale=1.3,>=stealth]
\node (P1) at (0,0)  {$1$};
\node (P2) at (1,1) {$2$};
\node(P3) at (2,0) {$3$};
\node(P4) at (3,1){$4$};
\node(P5) at (4,0){};
\node (P6) at (6,0){};
\node (P7) at (7,1) {$2w$};
\node (P8) at (8,0){$2w+1$};
\node (P9) at (8.5,1){$2w+2$};
\node (P10) at (10.5,1){};
\node (P11) at (12,1){$p+1$};
\node (P12) at (7,-1){$p+2$};
\node (P13) at (5.5,-1){};
\node(P14) at (1,-1){$p+q$};
\node(P15) at (2.5,-1){};

\draw [->] (P2)--(P1);
\draw[->] (P3)--(P2);
\draw[->] (P4)--(P3);
\draw[->] (P5)--(P4);
\draw[loosely dotted, thick] (P5)--(P6);
\draw[->] (P7)--(P6);
\draw[->] (P8)--(P7);
\draw[->] (P7)--(P9);
\draw[loosely dotted, thick,draw=red] (P9)--(P10);
\draw[->] (P10)--(P11);
\draw[->] (P12)--(P8);
\draw[->] (P13)--(P12);
\draw[loosely dotted, thick] (P15)--(P13);
\draw[->] (P14)--(P15);

\draw[->](P1)--(P14);
 \draw [loosely dotted, thick] (P1) .. controls (1,.8) .. (P3);
 \draw [loosely dotted, thick] (P3) .. controls (3,.8) .. (P5);
 \draw [loosely dotted, thick] (P6) .. controls (7,.8) .. (P8);

\node (P9+) at (10,1) {};
\draw [->,draw=blue] (P9) -- (P9+);
\draw [loosely dotted, thick,draw=blue] (P9+) -- (P10);
\end{tikzpicture}}\] 

Then for $0\leq j\leq w-1$, the projective resolution of $S_{2j+1}$ is given by 
\[\xymatrix{0\ar[r] & P_{2j+3}\ar[r] & P_{2j+2}\ar[r] & P_{2j+1}\ar[r] & S_{2j+1}\ar[r] & 0}.\] Thus one easily checks the following
\begin{equation}\label{eq1} \textrm{for } 0\leq j\leq w-1\quad \tau [S_{2j+1}]=[S_{2j+3}[1]]=-[S_{2j+3}] \quad\textrm{in } K_0(\Dd^b(\Lambda)).\end{equation}
We also have the projective resolutions
\[ \xymatrix{0\ar[r] & P_{p+2}\ar[r] & P_{2w+1}\ar[r] & S_{2w+1}\ar[r] & 0} \quad\textrm{and }\] \[\textrm{for }p+2\leq j\leq p+q \quad\xymatrix{0\ar[r] & P_{j+1}\ar[r] & P_{j}\ar[r] & S_{j}\ar[r] & 0},\]
where we use the convention $p+q+1=1$.
Hence we have 
\begin{equation}\label{eq2} \tau [S_{2w+1}]=[S_{p+2}] \quad\textrm{and for } p+2\leq j \leq p+q \quad \tau[S_{j}]=[S_{j+1}]\quad\textrm{in } K_0(\Dd^b(\Lambda)).\end{equation}
Combining \eqref{eq1} and \eqref{eq2} we get
\[ \tau^{q+w}[S_1]=(-1)^w\tau^q[S_{2w+1}]=(-1)^w[S_1]\]

Similarly we have
\begin{equation}\label{eq3} \textrm{for } 2w+3\leq j\leq p+1\quad \tau[S_j]=[S_{j-1}]\quad\textrm{in } K_0(\Dd^b(\Lambda)).\end{equation}
Now we have to separate the case where $w=1$. Assume $w\geq 2$ then we have
\begin{equation}\label{eq4} \tau[S_{2w+2}]=[P_{2w-2}]\quad\textrm{in } K_0(\Dd^b(\Lambda)).\end{equation}
\begin{equation}\label{eq5} \textrm{for }2\leq j\leq w-1 \quad \tau[P_{2j}]=-[I_{2j}]=-[P_{2j-2}] \quad\textrm{in } K_0(\Dd^b(\Lambda)).\end{equation}
Hence if $p+1\geq 2w+2$, then we have $I_2\iso P_{p+1}$ and we get the following equalities in $K_0(\Dd^b(\Lambda))$:
\begin{align*} \tau^{p-w}[P_2] & =-\tau^{p-w-1}[P_{p+1}]=\tau^{p-w+2}[S_{p+1}] & \\
& =\tau^{w-1}\tau^{p-2w-1}[S_{p+1}]=\tau^{w-1}[S_{2w+2}] & \textrm{by }\eqref{eq3} 
\\ & =\tau^{w-2}[P_{2w-2}] &\textrm{by }\eqref{eq4} \\ &=(-1)^{w}[P_2] &\textrm{by }\eqref{eq5} \end{align*}

If $p+1=2w+1$, we have $I_2\iso P_{2w}$, and we also get $\tau^{p-w}[P_2]=(-1)^w[P_2]$.

Assume that $w$ is odd. Then one can checks that the set
\[\{[S_{2j+1}], 0\leq j \leq w\}\cup\{[S_j], 2w+2\leq j\leq p+q\}\cup \{ [P_{2j}], 1\leq j\leq w-1\}\cup \{[I_2]\}\]
is a basis of $K_0(\Dd^b(\Lambda))$. Therefore the Coxeter matrix is diagonalizable and $C(X)=(X^{q+w}+1)(X^{p-w}+1)$.

Assume that $w$ is even. Then the set 
\[\{[S_{2j+1}], 0\leq j \leq w\}\cup\{[S_j], 2w+2\leq j\leq p+q\}\cup \{ [P_{2j}], 1\leq j\leq w-1\}\]
is linearly independent in $K_0(\Dd^b(\Lambda)$ and we have the relation:
\[ [I_2]- \sum_{j=2w+2}^{p+1}[S_j] -\sum_{j=1}^{w-1}(-1)^j[P_{2j}]=\sum_{j=p+2}^{p+q}[S_j]+\sum_{j=1}^{w}(-1)^j[S_{2j+1}]\]
This element is an eigenvector of the eigenvalue $1$.
Hence
\[ \frac{(X^{q+w}-1)(X^{p-w}-1)}{(X-1)} \]
divides the Coxeter polynomial $C(X)$. Since the degree of $C(X)$ is $p+q$, and since we know that both the leading coefficient and the absolut term of $C(X)$ are $1$, it follows that
\[ C(X) = \frac{(X^{q+w}-1)(X^{p-w}-1)}{(X-1)} \cdot (X-1) = (X^{q+w}-1)(X^{p-w}-1). \]

For the case $w=1$, we introduce the notations for $4\leq j\leq p$:
\[M_j:=\Ker (I_2\rightarrow I_{j+1}\oplus I_{j+1})\iso \Coker (P_2\rightarrow P_j\oplus P_j); \] \[ M_3:=\Ker (I_2\rightarrow I_{4}\oplus I_{4})\iso P_2;\quad \textrm{and }M_{p+1}:=\Coker  (P_2\rightarrow P_{p+1}\oplus P_{p+1})\iso I_2.\]
Then we have
\begin{equation}\label{eq6} \textrm{for } 4\leq j\leq p+1\quad \tau[M_j]=[M_{j-1}] \end{equation}
Therefore we get the equalities in $K_0(\Dd^b(\Lambda)):$
\[ \tau^{p-1}[P_2]=-\tau^{p-2}[I_2]=\tau^{p-2}[M_{p+1}]=-[M_3]=-[P_2]\]
Finally it is easy to see that the set 
\[ \{[M_j], 3\leq j\leq p+1\}\cup \{ [S_1], [S_3]\}\cup \{[S_j], p+2\leq j\leq p+q\}\]
is a basis of $\mathbb{Q}\ten_{\mathbb{Z}}K_0(\Dd^b(\Lambda))$. This finishes the proof.
\end{proof}

We end this subsection by giving a result linking the weight of an algebra of cluster type $\widetilde{A}_{p,q}$ with the set of cluster-tilting objects of $\Cc_{\widetilde{A}_{p,q}}$ coming from tilting complexes in $\Dd^b(\Lambda)$. We start with some notation.

Let $\Lambda$ be an algebra of global dimension at most 2 which is $\tau_2$-finite. We define a subset $\Tt_{\Lambda}$ of the set of tilting complexes of $\Dd^b(\Lambda)$ by \[ \Tt_\Lambda:=\{ T\in\Dd^b(\Lambda) \textrm{ tilting }\mid \ \textrm{gldim} (\End_{\Dd^b(\Lambda)}(T))\leq 2 \}\]
By Theorem~\ref{clustertilting}, the set $\pi_{\Lambda}(\Tt_{\Lambda}) \subset \Cc_\Lambda$ is a subset of the set of cluster-tilting objects of $\Cc_{\Lambda}$.  Moreover, if $\Lambda_1$ and $\Lambda_2$ are derived equivalent, they are cluster equivalent and we clearly have $\pi_{\Lambda_1}(\Tt_{\Lambda_1})=\pi_{\Lambda_2}(\Tt_{\Lambda_2})$.
We prove here the converse in the case of algebras of cluster type $\widetilde{A}_{p,q}$, and compare the sets $\pi_{\Lambda_1}(\Tt_{\Lambda_1})$ and $\pi_{\Lambda_2}(\Tt_{\Lambda_2})$ when $w(\Lambda_1)\neq w(\Lambda_2)$.
\begin{prop}
Let $\Lambda_1$ and $\Lambda_2$ be two algebras of global dimension 2 and of cluster type $\widetilde{A}_{p,q}$. Then we have 
\begin{itemize}
\item $0\leq w(\Lambda_1)<w(\Lambda_2) \Rightarrow \pi_2(\Tt_{\Lambda_2})\subsetneq \pi_1(\Tt_{\Lambda_1})$ 
\item $w(\Lambda_2)<w(\Lambda_1)\leq 0 \Rightarrow \pi_2(\Tt_{\Lambda_2})\subsetneq \pi_1(\Tt_{\Lambda_1})$
\item $w(\Lambda_1)w(\Lambda_2)<0 \Rightarrow \pi_1(\Tt_{\Lambda_1}) \setminus \pi_2(\Tt_{\Lambda_2}) \neq \emptyset \neq \pi_2(\Tt_{\Lambda_2}) \setminus \pi_1(\Tt_{\Lambda_1})$
\end{itemize}
\end{prop}

\begin{proof}
We first show the inclusion in the first claim.
Without loss of generality, we can assume that $w(\Lambda_2)=w(\Lambda_1)+1>0$.  It is enough to show that $\pi_2(\Lambda_2)\in\pi_1(\Tt_{\Lambda_1})$.
We denote by $H$ some hereditary algebra of type $\A_{p,q}$ and by $\pi_2$ a triangle functor $\pi_2 \colon \Dd^b(\Lambda_2)\rightarrow \Cc_H$.

Let $(Q,W_{Q},d)$ be a graded quiver with potential such that we have isomorphisms of $\ZZ$-graded algebras $\End_{\Cc_H}(\pi_2\Lambda_2)\iso \Jac(Q,W_{Q},d)$.  Since $w(\Lambda_2)\geq 1$, there exists a $p$-arrow $a_i\in Q$ such that $d(a_i)=1$. Since $d$ is a $W_{Q}$-grading,  the subquiver $Q^{a_i}$ is not empty. More precisely, there exists arrows $a_i'$ and $a_i''$ such that $a_ia_i'a_i''$ is an oriented triangle in $Q$.

Define a new degree $d'$ on $Q$ by:
\[ d'(x)=\left\{\begin{array}{cc} 
0 & \ \textrm{if } x=a_i \\
1& \textrm{if } x=a'_i\\
d(x) & \textrm{otherwise}
\end{array}\right. \]
It is immediate to see that $d'$ is a $W_{Q}$-grading.  Define the algebra $\Lambda_3$ as the degree 0 part of the graded Jacobian algebra $\Jac(Q,W_{Q},d')$. By Theorem~\ref{algebraclustertypetildeA}, it is an algebra of global dimension 2 which is of cluster type $\A_{p,q}$, and by Corollary~\ref{recognitioncor} we can assume $\pi_2(\Lambda_2)=\pi_3(\Lambda_3)$ where $\pi_3$ is a triangle functor $\pi_3 \colon \Dd^b(\Lambda_3)\rightarrow \Cc_{\Lambda_3}$.  Moreover by Proposition~\ref{prop alternative description weight}, we have $w(\Lambda_3)=w(\Lambda_2)-1=w(\Lambda_1)$. Therefore by Theorem~\ref{derivedeqiffwegal} the algebra $\Lambda_3$ is derived equivalent to $\Lambda_1$. The image of $\Lambda_3$ through this equivalence is clearly an object $X$ in $\Tt_{\Lambda_1}$ which satisfies $\pi_1(X)\iso \pi_2(\Lambda_2)$.

For any $0\leq w \leq [\frac{p}{2}]$, one can easily construct a quiver $Q$ such that $Q$ admits a $W_Q$-grading of weight $w$ but no $W_Q$-grading of weight $w+1$. Therefore the inclusion is strict. 

The second point holds by symmetry.

For the third point, it is enough to see that for any $w>0$ the quiver $Q$ constructed in the proof of Corollary~\ref{corollary cycle} satisfies the following: there exists a $W_Q$-grading of weight $w$ and for any $w'<0$ there is no $W_Q$-grading of weight $w'$. The argument for $w<0$ holds by symmetry.
\end{proof}

We end this subsection by asking the following intriguing questions:

\begin{Question}
\begin{itemize}
\item
Let $\Lambda_1$ and $\Lambda_2$ be $\tau_2$-finite algebras of global dimension $\leq 2$ which are cluster equivalent. Do we have the implication
\[ \pi_1(\Tt_{\Lambda_1})=\pi_2(\Tt_{\Lambda_2})\Rightarrow \Dd^b(\Lambda_1)\iso\Dd^b( \Lambda_2) \ ?\]
\item
Let $\Lambda$ be a $\tau_2$-finite algebra of global dimension $\leq 2$. Does the following implication hold?
\[ \pi_{\Lambda}(\Tt_\Lambda)=\{ X\in \Cc_{\Lambda} \mid X \textrm{ cluster-tilting }\}\Rightarrow \Lambda \textrm{ piecewise hereditary }\]
\end{itemize} 
\end{Question}

\subsection{Example}

In this subsection we compute explicitly all basic algebras (up to isomorphism) of global dimension $\leq 2$ and of cluster type $\A_{2,2}$ and organize them according to their derived equivalence classes.

The strategy consists in first  describing all quivers (up to isomorphism of quivers) which are in the set $\Mm^{\A}_{2,2}$. In our case an easy computation or \cite{Keller_Java} shows that there are only 4 different quivers in $\Mm^{\A}_{2,2}$ which are:
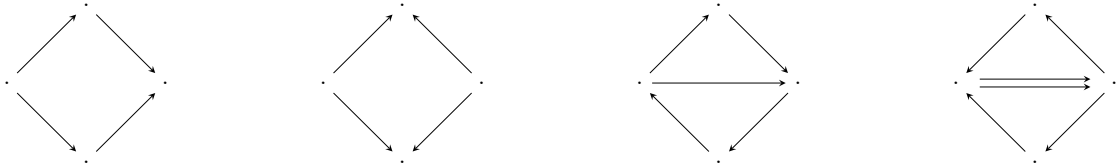
\[\scalebox{.8}{
\begin{tikzpicture}[scale=1.3,>=stealth]
\node (A1) at (0,0) {$.$};
\node (A2) at (1,1){$.$};
\node (A3) at (2,0){$.$};
\node (A4) at (1,-1){$.$};

\draw [->] (A1)--(A2);
\draw [->] (A2)--(A3);
\draw [->] (A1)--(A4);
\draw [->] (A4)--(A3);

\node (B1) at (4,0) {$.$};
\node (B2) at (5,1){$.$};
\node (B3) at (6,0){$.$};
\node (B4) at (5,-1){$.$};

\draw [->] (B1)--(B2);
\draw [->] (B3)--(B2);
\draw [->] (B1)--(B4);
\draw [->] (B3)--(B4);

\node (C1) at (8,0) {$.$};
\node (C2) at (9,1){$.$};
\node (C3) at (10,0){$.$};
\node (C4) at (9,-1){$.$};

\draw [->] (C1)--(C2);
\draw [->] (C2)--(C3);
\draw [->] (C3)--(C4);
\draw [->] (C4)--(C1);
\draw [->] (C1)--(C3);

\node (D1) at (12,0) {$.$};
\node (D2) at (13,1){$.$};
\node (D3) at (14,0){$.$};
\node (D4) at (13,-1){$.$};

\draw [->] (D2)--(D1);
\draw[->] (D3)--(D2);
\draw [->] (D3)--(D4);
\draw [->] (D4)--(D1);
\draw [->] (12.3,.05)--(13.7,.05);
\draw [->] (12.3,-.05)--(13.7,-.05);
\end{tikzpicture}}\]

Note that since $p=q$ there is an isomorphism between the two quivers corresponding to $p_r=2, q_r=1$ and $p_r=1,q_r=2$
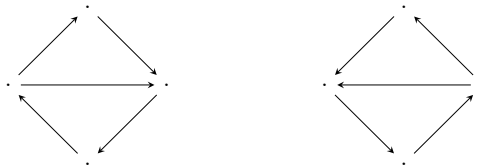
\[\scalebox{.8}{
\begin{tikzpicture}[scale=1.3,>=stealth] \node (C1) at (8,0) {$.$};
\node (C2) at (9,1){$.$};
\node (C3) at (10,0){$.$};
\node (C4) at (9,-1){$.$};

\draw [->] (C1)--(C2);
\draw [->] (C2)--(C3);
\draw [->] (C3)--(C4);
\draw [->] (C4)--(C1);
\draw [->] (C1)--(C3);

\node (D1) at (12,0) {$.$};
\node (D2) at (13,1){$.$};
\node (D3) at (14,0){$.$};
\node (D4) at (13,-1){$.$};

\draw [<-] (D1)--(D2);
\draw [<-] (D2)--(D3);
\draw [<-] (D3)--(D4);
\draw [<-] (D4)--(D1);
\draw [<-] (D1)--(D3);
 \end{tikzpicture}}\]
Then one can easily check that there are 11 graded quivers $(Q,d)$ with $Q\in\Mm^{\A}_{p,q}$ and $d$ a $W_Q$-grading up to isomorphism of graded quiver. Therefore there are exactly 11 non-isomorphic algebras of global dimension 2 and of cluster type $\A_{2,2}$.

The only possible weights are $|w|=0$ or $|w|=1$. Therefore these eleven algebras are divided into two derived equivalence classes.

There are 8 algebras of global dimension $\leq 2$ which are derived equivalent to $\A_{2,2}$:
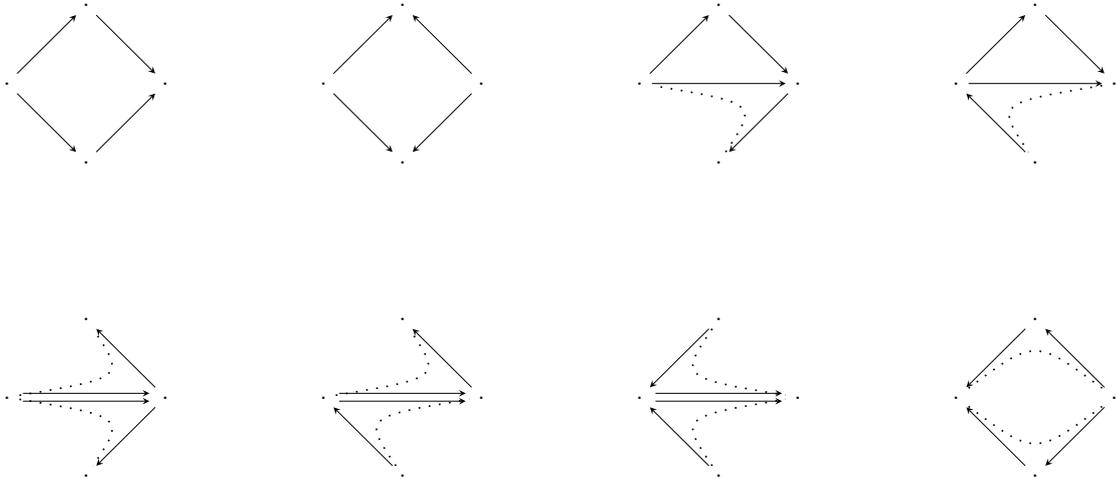
\[\scalebox{.8}{
\begin{tikzpicture}[scale=1.3,>=stealth]
\node (A1) at (0,0) {$.$};
\node (A2) at (1,1){$.$};
\node (A3) at (2,0){$.$};
\node (A4) at (1,-1){$.$};

\draw [->] (A1)--(A2);
\draw [->] (A2)--(A3);
\draw [->] (A1)--(A4);
\draw [->] (A4)--(A3);

\node (B1) at (4,0) {$.$};
\node (B2) at (5,1){$.$};
\node (B3) at (6,0){$.$};
\node (B4) at (5,-1){$.$};

\draw [->] (B1)--(B2);
\draw [->] (B3)--(B2);
\draw [->] (B1)--(B4);
\draw [->] (B3)--(B4);

\node (C1) at (8,0) {$.$};
\node (C2) at (9,1){$.$};
\node (C3) at (10,0){$.$};
\node (C4) at (9,-1){$.$};

\draw [->] (C1)--(C2);
\draw [->] (C2)--(C3);
\draw [->] (C3)--(C4); 
\draw [loosely dotted,thick] (C4)..controls (9.5,-0.25)..(C1);
\draw [->] (C1)--(C3);

\node (D1) at (12,0) {$.$};
\node (D2) at (13,1){$.$};
\node (D3) at (14,0){$.$};
\node (D4) at (13,-1){$.$};

\draw[->] (D1)--(D2);
\draw[->] (D2)--(D3);
\draw[->] (D1)--(D3);
\draw[->] (D4)--(D1);
\draw [loosely dotted, thick] (D3)..controls (12.5,-.25)..(D4);

\node (E1) at (0,-4) {$.$};
\node (E2) at (1,-3){$.$};
\node (E3) at (2,-4){$.$};
\node (E4) at (1,-5){$.$};

\draw[->] (0.2,-3.95)--(1.8,-3.95);
\draw[->] (0.2,-4.05)--(1.8,-4.05);
\draw[->] (E3)--(E2);
\draw[->] (E3)--(E4);
\draw [loosely dotted, thick] (E1)..controls(1.5,-3.75)..(E2);
\draw [loosely dotted, thick] (E1)..controls(1.5,-4.25)..(E4);

\node (F1) at (4,-4) {$.$};
\node (F2) at (5,-3){$.$};
\node (F3) at (6,-4){$.$};
\node (F4) at (5,-5){$.$};

\draw[->] (4.2,-3.95)--(5.8,-3.95);
\draw[->] (4.2,-4.05)--(5.8,-4.05);
\draw[->] (F3)--(F2);
\draw[->] (F4)--(F1);
\draw [loosely dotted, thick] (F1)..controls(5.5,-3.75)..(F2);
\draw [loosely dotted, thick] (F4)..controls(4.5,-4.25)..(F3);

\node (G1) at (8,-4) {$.$};
\node (G2) at (9,-3){$.$};
\node (G3) at (10,-4){$.$};
\node (G4) at (9,-5){$.$};

\draw[->] (8.2,-3.95)--(9.8,-3.95);
\draw[->] (8.2,-4.05)--(9.8,-4.05);
\draw[->] (G2)--(G1);
\draw[->] (G4)--(G1);
\draw [loosely dotted, thick] (G2)..controls (8.5,-3.75)..(G3);
\draw [loosely dotted, thick] (G4)..controls(8.5,-4.25)..(G3);

\node (H1) at (12,-4) {$.$};
\node (H2) at (13,-3){$.$};
\node (H3) at (14,-4){$.$};
\node (H4) at (13,-5){$.$};

\draw[->] (H2)--(H1);
\draw[->] (H3)--(H2);
\draw[->] (H4)--(H1);
\draw[->] (H3)--(H4);
\draw [loosely dotted, thick] (H1)..controls (13,-3.25)..(H3);
\draw [loosely dotted, thick] (H1)..controls (13,-4.75)..(H3);

\end{tikzpicture}}\]

There are 3 algebras of global dimension $\leq 2$ which are of cluster type $\A_{2,2}$ and not derived equivalent to $\A_{2,2}$. They are all derived equivalent to each other and not piecewise hereditary:

\[\scalebox{.8}{
\begin{tikzpicture}[scale=1.3,>=stealth]
\node (A1) at (0,0) {$.$};
\node (A2) at (1,1){$.$};
\node (A3) at (2,0){$.$};
\node (A4) at (1,-1){$.$};

\draw [->] (A1)--(A2);
\draw [->] (A2)--(A3);
\draw [->] (A3)--(A4);
\draw [->] (A4)--(A1);
\draw [loosely dotted, thick]  (A1)..controls (1,0.75)..(A3);

\node (B1) at (4,0) {$.$};
\node (B2) at (5,1){$.$};
\node (B3) at (6,0){$.$};
\node (B4) at (5,-1){$.$};

\draw [->] (B2)--(B1);
\draw [->] (B3)--(B2);
\draw [->] (B1)--(B3);
\draw [->] (B3)--(B4);
\draw [loosely dotted, thick] (B1)..controls (5,0.75)..(B3);
\draw[loosely dotted, thick] (B1)..controls (5.5,-.25)..(B4);

\node (C1) at (8,0) {$.$};
\node (C2) at (9,1){$.$};
\node (C3) at (10,0){$.$};
\node (C4) at (9,-1){$.$};

\draw [->] (C2)--(C1);
\draw [->] (C3)--(C2);
\draw [->] (C3)--(C1); 
\draw [loosely dotted,thick] (C4)..controls (8.5,-0.25)..(C3);
\draw [->] (C4)--(C1);
\draw[loosely dotted, thick] (C1)..controls (9,0.75)..(C3);
\end{tikzpicture}}\]

\def\cprime{$'$}
\newcommand{\etalchar}[1]{$^{#1}$}
\providecommand{\bysame}{\leavevmode\hbox to3em{\hrulefill}\thinspace}
\providecommand{\MR}{\relax\ifhmode\unskip\space\fi MR }
\providecommand{\MRhref}[2]{%
  \href{http://www.ams.org/mathscinet-getitem?mr=#1}{#2}
}
\providecommand{\href}[2]{#2}

\end{document}